\documentclass[11pt]{article}
\usepackage[margin=1in]{geometry}
\usepackage{amsmath,amssymb,amsthm}
\usepackage{graphicx}
\usepackage[colorlinks=true, linkcolor=blue, citecolor=blue, urlcolor=blue]{hyperref}
\usepackage{cleveref}
\usepackage{bm}
\usepackage{subcaption}
\usepackage{booktabs}
\graphicspath{{images_app1/}{Images/}}
\usepackage{algorithm}
\usepackage{amsfonts}
\usepackage{stmaryrd}
\usepackage{dsfont}
\usepackage{xfrac}
\usepackage{mathrsfs}
\usepackage{algpseudocode}
\usepackage{adjustbox}
\usepackage[english]{babel}
\usepackage{csquotes}
\usepackage[
  sortcites=true,
  style=numeric-comp,
  backend=biber,
  sorting=nyt,
  maxnames=20,
  minnames=1,
  giveninits=true,
  doi=true,
  url=true
]{biblatex}
\usepackage{tikz}
\usepackage{pgfplots}
\usepackage{pgfplotstable}
\pgfplotsset{compat=1.18}
\usepgfplotslibrary{groupplots,units,colormaps}
\usetikzlibrary{arrows.meta}

\theoremstyle{remark}
\newtheorem{remark}{Remark}[section]
\theoremstyle{plain}
\newtheorem{theorem}{Theorem}[section]
\newtheorem{lemma}[theorem]{Lemma}
\newtheorem{proposition}[theorem]{Proposition}
\newtheorem{assumption}{Assumption}

\newcommand{\R}{\mathbb{R}}
\newcommand{\bs}[1]{\boldsymbol{#1}}

\newcommand{\dx}{\,\mathrm{d}\mathbf{x}}
\newcommand{\bx}{\mathbf{x}}
\newcommand{\bn}{\mathbf{n}}

\newcommand{\bX}{\mathbf{X}}
\newcommand{\bxi}{\boldsymbol{\xi}}

\newcommand{\Hrel}{H_{\rm rel}}
\newcommand{\Ieff}{I_{\rm eff}}

\newcommand{\Gammain}{\Gamma_{\mathrm{in}}}
\newcommand{\Gammaout}{\Gamma_{\mathrm{out}}}
\newcommand{\btheta}{\boldsymbol{\theta}}
\newcommand{\norm}[1]{\left\lVert #1 \right\rVert}

\makeatletter
\let\@fnsymbol\@arabic
\makeatother
\begin{document}
\title{Parametric Neural $r$-Adaptivity for Isogeometric Analysis via Residual Minimization}
\author{
  El\'ias Caru\textsuperscript{2,1} \and
  David Pardo\textsuperscript{1,2,3} \and
  Judit Mu\~noz-Matute\textsuperscript{1,3}
}
\date{}

\maketitle

\footnotetext[1]{Universidad del Pa\'is Vasco/Euskal Herriko Unibertsitatea (UPV/EHU), Leioa, Spain}
\footnotetext[2]{Basque Center for Applied Mathematics (BCAM), Bilbao, Spain}
\footnotetext[3]{Basque Foundation for Science (Ikerbasque), Bilbao, Spain}

\begin{abstract} 

\noindent We propose an $r$-adaptive neural algorithm for Isogeometric Analysis (IGA) based on residual minimization. The boundary-value problem is solved using a standard conforming Galerkin formulation, while a neural network relocates the interior knots. A strong-form residual in the sense of physics-informed neural networks (PINNs) controls a norm stronger than the energy ($H^1$) error. We therefore weight it by classical \textit{a posteriori} theory: element residuals scaled by the local mesh size, interface flux jumps, and Neumann boundary residuals yield a computable estimator of the energy error, which we minimize with respect to the knots. For coercive problems on admissible mesh families, this estimator is reliable and locally efficient up to oscillation terms; beyond that regime, the same loss remains well-defined and extends differentiable $r$-adaptivity to indefinite and advection-dominated problems. In the parametric setting, the network maps each parameter to a knot-density function in a single evaluation; since it outputs a density rather than a fixed-dimensional vector of knot locations, one trained network produces an admissible mesh at any refinement level. Mesh gradients are obtained by reverse-mode automatic differentiation through the discrete solution equation. Numerical experiments in one and two dimensions illustrate that the method concentrates degrees of freedom near singularities, material interfaces, and boundary layers, improving accuracy for a fixed number of degrees of freedom.
\end{abstract}

\section{Introduction}
\label{sec:intro}

Solutions of partial differential equations (PDEs) often exhibit localized features: sharp gradients, boundary layers, material interfaces, or corner singularities. Uniform refinement resolves such features at a high computational cost; adaptive methods instead introduce degrees of freedom (DOFs) where they are required. Classical $h$- and $p$-adaptivity modify the mesh size or the polynomial degree, while $r$-adaptivity keeps the number of DOFs fixed and redistributes the mesh points~\cite{bangerth2003adaptive,huang2010adaptive,budd2009adaptive}. 

Isogeometric Analysis (IGA)~\cite{hughes2005isogeometric,bazilevs2006isogeometric,cottrell2009isogeometric} provides a natural setting for $r$-adaptivity: its discrete spaces are constructed from knot vectors, and relocating interior knots modifies the local resolution without altering the dimension of the space or the tensor-product structure. Most existing $r$-adaptive IGA methods rely on monitor functions or mesh-quality criteria, such as Winslow mappings~\cite{xu2019efficient}. Artificial neural networks have also been employed in this context~\cite{rios2024adaptive}: a network relocates the inner control points of a multi-patch parametrization, trained against a mesh-quality measure. In such approaches, the mesh is optimized with respect to a geometric criterion rather than the error of the discrete solution it produces.

In this work, a neural network determines the interior knot positions, while the discrete solution is computed by a standard conforming Galerkin IGA solver; the network does not replace the solver. Training the network by gradient descent requires a differentiable loss function that quantifies the error of this Galerkin solution. The appropriate requirement is equivalence: the loss should be equivalent to the energy ($H^1$) error of the discrete solution, so that minimizing the former reduces the latter. For symmetric positive definite (SPD) problems, the Ritz energy satisfies this requirement and is the natural choice. For non-SPD problems, the natural candidate is the strong-form PDE residual minimized by physics-informed neural networks (PINNs)~\cite{raissi2019physics}. This residual, however, is not equivalent to the $H^1$ error: it controls a norm stronger than the energy norm, as it involves the highest-order derivatives of the error rather than first-order ones. A mesh optimized under this loss is consequently graded for the wrong error measure. The mismatch is most pronounced precisely where adaptivity is most needed: singular solutions lack the regularity required for the strong residual to be well defined, and the loss cannot certify convergence in the norm of interest.

We resolve this norm mismatch through classical \textit{a posteriori} error estimation~\cite{ainsworth2000posteriori,verfurth2013posteriori}. Weighting each element residual by the local mesh size and incorporating the interface flux jumps and boundary residuals yields a fully computable estimator of the energy error, which is reliable and locally efficient in the coercive regime~\cite{rojas2021goaloriented,los2021igrm}, evaluable by standard element-wise quadrature, and differentiable with respect to the knot positions. This estimator constitutes our training loss and the resulting $r$-adaptivity applies beyond symmetric coercive problems, including the indefinite and advection-dominated cases considered in this paper.

Alternative differentiable loss functions address the mismatch only partially. Ritz-based $r$-adaptivity minimizes the discrete energy, which targets the correct norm but presupposes a minimization principle, restricting the approach to symmetric coercive problems~\cite{magueresse2025energy,aballay2025radaptivefiniteelementmethod}. Dual-norm residual losses~\cite{kharazmi2021hp,taylor2023deepfourier,uriarte2023doubleritz,rojas2024rvpinn,taylor2025dfrmaxwell} also target the correct norm, but the dual norm is not directly computable: it is defined as a supremum over an infinite-dimensional test space and therefore cannot be evaluated exactly. In practice, the supremum is approximated by its restriction to a finite-dimensional test space. This approximation introduces additional considerations: the accuracy of the resulting loss depends on the choice of the discrete test space, the inversion of the associated Gram matrix is required, and the loss represents the true dual norm only up to the test-space discretization and its numerical integration. Learned mesh-movement networks~\cite{song2022m2n,foucart2023drlamr} are trained on supervision or reward signals rather than on an error estimate, and neural solvers with moving meshes~\cite{omella2024radapt} represent the solution in a nonlinear trial space, forgoing the approximation guarantees of a conforming Galerkin method.

In the parametric setting, a single network is sought that predicts an adapted mesh for every problem parameter, avoiding a separate optimization per instance; this is the setting of~\cite{aballay2025radaptivefiniteelementmethod}, where the network regresses the node positions from the problem parameter. A separate question is how the mesh is represented at the network output. A fixed-length vector of node positions ties the architecture to a single refinement level: a mesh with a different number of elements requires a new network and a new training procedure. We instead predict a knot \emph{density}: a continuous function over the domain whose sampling at any resolution yields an admissible mesh, so that one trained network serves every refinement level and enables the coarse-to-fine continuation strategy employed throughout.

The main contributions of this paper are twofold.
\begin{enumerate}
\item We employ a classical \textit{a posteriori} estimator of the energy ($H^1$) error as the training loss for $r$-adaptivity: element residuals weighted by the local mesh size, together with interface flux jumps and boundary residuals. We prove that this loss is reliable and locally efficient in the coercive regime, and its construction extends $r$-adaptivity beyond symmetric coercive problems.
\item We introduce the \emph{residual-informed neural mesh} for the parametric setting: the network outputs a knot \emph{density} rather than knot positions, which decouples the trained network from the refinement level.
\end{enumerate}

The implementation is differentiable end to end: exact gradients of the loss with respect to the knots are obtained through the discrete adjoint, which reverse-mode automatic differentiation realizes on the linear solve at the cost of one additional solve. The method is validated on five parametric benchmarks, including an indefinite Helmholtz transmission problem, boundary layers, and a re-entrant corner. Regarding the scope of this work, the results are restricted to coercive diffusion--reaction and advection--diffusion--reaction problems, for which the conforming Galerkin method is stable; the boundary-layer and Helmholtz examples are considered as fixed-DOF error-reduction tests and do not imply robustness with respect to the perturbation parameter. The geometry map remains fixed throughout: only the knot distribution inside the domain is adapted, not the shape of the domain.

The remainder of the paper is organized as follows. Section~\ref{sec:problem_setting} introduces the model problem and the IGA discretization. Section~\ref{sec:proposed_approach} defines the residual loss and the discrete-adjoint gradient. Section~\ref{sec:parametric_inn} extends the method to parametric problems through the residual-informed neural mesh. Section~\ref{sec:numerical_examples} presents the numerical experiments, and Section~\ref{sec:conclusions} draws the conclusions.

\section{Problem setting}
\label{sec:problem_setting}

\subsection{Linear PDE model problem}
Let \(\Omega\subset \mathbb{R}^d\) (\(d\in\{1,2\}\)) be an open bounded Lipschitz domain with boundary \(\partial\Omega=\overline{\Gamma}_D\cup\overline{\Gamma}_N\) and \(\Gamma_D\cap\Gamma_N=\emptyset\), where \(\Gamma_D\) and \(\Gamma_N\) denote the Dirichlet and Neumann parts, respectively (\(\Gamma_N\) possibly empty). We are given coefficient fields
\[
\sigma \in L^\infty(\Omega), \qquad 
\bs{\beta}\in [W^{1,\infty}(\Omega)]^d, \qquad 
\alpha\in L^\infty(\Omega),
\]
with \(\sigma\) uniformly positive and piecewise \(W^{1,\infty}\), 
i.e., there exist constants \(0<\sigma_{\min}\le\sigma_{\max}\) 
such that
\[
\sigma_{\min}\le \sigma(\bx)\le \sigma_{\max} 
\quad \text{for a.e.}\ \bx\in\Omega.
\]
Moreover, let \(f:\Omega\to\R\), \(u_D:\Gamma_D\to\R\), and 
\(g:\Gamma_N\to\R\) denote the source term, Dirichlet data, and 
Neumann data, respectively. We consider the following boundary-value problem
\begin{equation}
\label{eq:general_pde}
\left\{
\begin{aligned}
- \nabla \cdot (\sigma \nabla u) + \bs{\beta}\cdot\nabla u + \alpha u &= f, && \text{in } \Omega,\\
u &= u_D, && \text{on } \Gamma_D,\\
\sigma \dfrac{\partial u}{\partial \bn} &= g, && \text{on } \Gamma_N.
\end{aligned}
\right.
\end{equation}
Here \(\bn\) denotes the outward unit normal vector on \(\partial\Omega\) and \(\partial_{\bn}u:=\nabla u\cdot \bn\). For advection--diffusion problems, the inflow and outflow boundaries are defined by
\[
\Gammain := \{\bx\in\partial\Omega:\ \bs{\beta}(\bx)\cdot \bn(\bx)<0\},
\qquad
\Gammaout := \partial\Omega\setminus\overline{\Gammain},
\]
and we assume \(\Gammain\subseteq \Gamma_D\).

\subsection{Weak formulation}
We assume \(f\in L^2(\Omega)\), \(g\in H^{-1/2}(\Gamma_N)\), and \(u_D\in H^{1/2}(\Gamma_D)\). Let \(\widetilde u_D\in H^1(\Omega)\) be a lifting of the Dirichlet data, i.e., \(\widetilde u_D|_{\Gamma_D}=u_D\) in the sense of traces, and define
\[
V:=H^1_{0,D}(\Omega)=\{v\in H^1(\Omega): v|_{\Gamma_D}=0\}.
\]
The standard variational formulation of \eqref{eq:general_pde} reads: find \(u_0\in V\) such that
\begin{equation}
\label{eq:weak_form}
B(u_0,v)=F(v)-B(\widetilde u_D,v) \qquad \forall\, v\in V,
\end{equation}
and set \(u:=u_0+\widetilde u_D\) (so that \(B(u,v)=F(v)\) for all \(v\in V\)), where
\[
B(w,v):=(\sigma\nabla w,\nabla v)_\Omega + (\bs{\beta}\cdot\nabla w,v)_\Omega + (\alpha w,v)_\Omega,
\qquad
F(v):=(f,v)_\Omega + \langle g,v\rangle_{\Gamma_N}.
\]
Here \((\cdot,\cdot)_\Omega\) denotes the \(L^2(\Omega)\) inner product and \(\langle\cdot,\cdot\rangle_{\Gamma_N}\) is the duality pairing between \(H^{-1/2}(\Gamma_N)\) and \(H^{1/2}(\Gamma_N)\). When \(g\in L^2(\Gamma_N)\), this pairing reduces to the boundary integral
\[
\langle g,v\rangle_{\Gamma_N} = \int_{\Gamma_N} g\, v\, \mathrm{d}\Gamma.
\]

\subsection{Mesh parametrization}
\label{sec:mesh_param}
We describe the one-dimensional construction on the reference domain \(\widehat\Omega=[0,1]\); the multidimensional case follows by tensor products of independent univariate partitions.

\paragraph{Fixed interfaces and segments.}
Let \(A=\{a_0<a_1<\cdots<a_S\}\) be ordered fixed points: \(a_0,a_S\) are the domain boundaries, and interior interfaces encode fixed constraints such as material interfaces or boundary-condition region separators. On each segment \([a_{s-1},a_s]\) of length \(L_s:=a_s-a_{s-1}\), we prescribe \(n_{\mathrm{el}}^{(s)}\ge 1\) elements, subject to \(L_s > n_{\mathrm{el}}^{(s)}\,h_{\min}\), where \(h_{\min}>0\) prevents element collapse.

\paragraph{Element-size parametrization.}
For each segment \(s\), a vector \(\bs\theta^{(s)}\in\R^{n_{\mathrm{el}}^{(s)}}\) is mapped by the softmax function to element-size proportions summing up to one, \(\bs\delta^{(s)} = \mathrm{softmax}(\bs\theta^{(s)})\), as in recent differentiable \(r\)-adaptive parametrizations~\cite{aballay2025radaptivefiniteelementmethod,magueresse2025energy}. In this non-parametric section, the \(\bs\theta^{(s)}\) are the optimization variables; in the parametric setting of Section~\ref{sec:parametric_inn}, they are produced by a neural network whose weights are the trainable quantities. The physical element sizes are
\[
h_i^{(s)}(\btheta) := h_{\min} + \bigl(L_s - n_{\mathrm{el}}^{(s)}\,h_{\min}\bigr)\,\delta_i^{(s)},
\qquad i=1,\ldots,n_{\mathrm{el}}^{(s)},
\]
a smooth map that enforces \(\sum_i h_i^{(s)} = L_s\) and \(h_i^{(s)} \ge h_{\min}\).

\begin{figure}[h!]
\centering
\includegraphics[width=\linewidth]{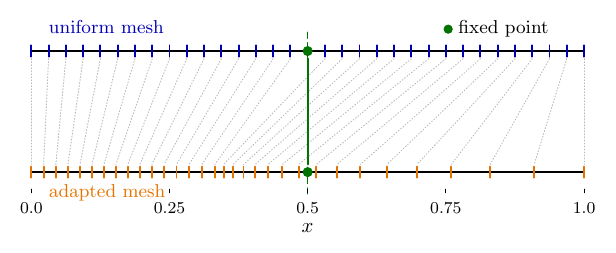}
\caption{Knot redistribution for the one-dimensional contrast Helmholtz problem (free-split treatment, Section~\ref{subsec:inn_mesh_predictor}): a uniform partition (top, blue) and the adapted partition (bottom, orange), with connectors tracing each interior breakpoint. The material interface at \(x=0.5\) (green) stays pinned at its prescribed multiplicity, while the number of elements on each side is governed by \(\btheta\).}
\label{fig:knot_redistribution}
\end{figure}

\subsection{Isogeometric spline space}
\label{sec:IGA_techniques}
Let \(K(\btheta)=\{k_j(\btheta)\}_{j=0}^{n_e}\) be the breakpoint vector of Section~\ref{sec:mesh_param}, \(p\ge 1\) the degree, and \(c\in\{0,\ldots,p-1\}\) the interior continuity. With interior knot multiplicity \(m:=p-c\), the open knot vector \(\Xi(\btheta)\) repeats each interior breakpoint \(m\) times and each endpoint \(p+1\) times. Repeating the endpoints (an open, or clamped, knot vector) makes the basis interpolatory at the boundaries---the first and last basis functions equal one there---so Dirichlet conditions can be imposed directly at the endpoint coefficients. The degree-\(p\) B-spline basis is then built from \(\Xi(\btheta)\) by the Cox--De Boor recursion~\cite{deboor2001practical,hughes2005isogeometric,cottrell2009isogeometric}, spanning \(\mathbb S_{p,c}(K(\btheta)):=\mathrm{span}\{B_i^p(\cdot;\btheta)\}_{i=0}^n\) of dimension \(n+1\), \(n = p + (n_e-1)m\). In \(d\) dimensions, tensor products give a basis \(\{\widehat R_{\mathbf i}(\cdot;\btheta)\}_{\mathbf i\in\mathcal I}\) on \(\widehat\Omega=[0,1]^d\). A fixed geometry map
\begin{equation}
\label{eq:geometry_map}
\bX:\widehat\Omega\to\Omega,
\qquad
\bx=\bX(\bxi),
\end{equation}
patchwise \(C^{1,1}\) (hence \(W^{2,\infty}\)) and bi-Lipschitz on patch interiors---in all experiments below, \(\bX\) is the identity---is given and kept fixed throughout: the reference breakpoints move during adaptation, but the physical boundary \(\partial\Omega\) does not. The physical mesh \(\mathcal T_h(\btheta)\) is the image of the parametric partition under \(\bX\), and with the push-forward \(N_{\mathbf i}(\bx;\btheta):=\widehat R_{\mathbf i}(\bX^{-1}(\bx);\btheta)\) the discrete spaces are
\[
V_h(\btheta):=\mathrm{span}\{N_{\mathbf i}(\cdot;\btheta)\}_{\mathbf i\in\mathcal I}\subset H^1(\Omega),
\qquad
V_h^0(\btheta):=V_h(\btheta)\cap V.
\]

\subsection{Galerkin discretization}
\label{sec:galerkin}
For a fixed admissible parameter $\btheta$, the Galerkin solution is $u_\theta=\widetilde u_{D,h}(\btheta)+u_\theta^0$, with $u_\theta^0\in V_h^0(\btheta)$, satisfying
\begin{equation}
\label{eq:discrete_weak_form}
B(u_\theta,v_h)=F(v_h)
\qquad \forall v_h\in V_h^0(\btheta).
\end{equation}
Restricting~\eqref{eq:weak_form} to the spline space $V_h^0(\btheta)$, we expand the unknown homogeneous component $u_\theta^0$ in the unconstrained basis functions and test with the same basis. The degrees of freedom associated with the Dirichlet data are fixed through the lifting $\widetilde u_{D,h}(\btheta)$, and their contributions are incorporated into the right-hand side. This gives the reduced algebraic system
\begin{equation}
\label{eq:linear_system}
\mathbf K(\btheta)\mathbf U(\btheta)=\mathbf F(\btheta).
\end{equation}
Here, for the basis functions $N_i$ associated with the free degrees of freedom,
\begin{equation}
\label{eq:matrix-dis}
\begin{aligned}
K_{ij}(\btheta)
&=\int_\Omega
\Bigl(
\sigma\nabla N_j(\btheta)\cdot\nabla N_i(\btheta)
+(\bs\beta\cdot\nabla N_j(\btheta))N_i(\btheta)
+\alpha N_j(\btheta)N_i(\btheta)
\Bigr)\dx,\\
F_i(\btheta)
&=(f,N_i(\cdot;\btheta))_\Omega
+\langle g,N_i(\cdot;\btheta)\rangle_{\Gamma_N}
-B\bigl(\widetilde u_{D,h}(\btheta),N_i(\cdot;\btheta)\bigr).
\end{aligned}
\end{equation}
The dependence on $\btheta$ enters through the basis functions, the quadrature points, and the element measures.

\section{\texorpdfstring{\(r\)-Adaptivity}{r-Adaptivity} via a 
differentiable residual-based objective}
\label{sec:proposed_approach}

\subsection{Residual-based loss}
\label{sec:loss}

The weak solution of~\eqref{eq:general_pde} may belong only to \(H^1(\Omega)\), so its strong-form residual is not globally defined in \(L^2(\Omega)\). The quantity of interest is instead the dual norm of the residual functional \(\mathfrak R_h(\btheta)\in V'\) of the discrete solution, which we estimate by a computable, mesh-dependent quantity built from strong residuals evaluated on the discrete spline solution. Since \(u_\theta\) is smooth inside each element, these residuals are well defined elementwise, and interface terms are added wherever the normal flux may be discontinuous. Residual minimization to drive mesh adaptation has been studied in the finite-element setting~\cite{rojas2021goaloriented,los2021igrm}; here it furnishes the objective for the knot positions.

\paragraph{Local residual components.}
We define the residual functional $\mathfrak R_h(\btheta)\in V'$ by
\begin{equation}
\label{eq:residual_functional}
\langle\mathfrak R_h(\btheta),v\rangle
:=F(v)-B(u_\theta,v),
\qquad v\in V.
\end{equation}
Galerkin orthogonality gives $\langle\mathfrak R_h(\btheta),v_h\rangle=0$ for all $v_h\in V_h^0(\btheta)$. For each element $E\in\mathcal T_h(\btheta)$, we define the strong residual as
\begin{equation}
\label{eq:element_residual}
\mathcal R_E[u_\theta]
:=f+\nabla\cdot(\sigma\nabla u_\theta)
-\bs\beta\cdot\nabla u_\theta
-\alpha u_\theta
\qquad \text{in }E.
\end{equation}
Let $\mathcal B_{\rm int}^\star(\btheta)$ be the set of interior faces where the normal flux may jump, for instance $C^0$ knot lines or material interfaces. If $I=\partial E^+\cap\partial E^-$ and $\bn^\pm$ are the outward normals of $E^\pm$, we use the sign convention
\begin{equation}
\label{eq:jump_residual}
\mathcal J_I[u_\theta]
:=-\bigl(\sigma^+\nabla u_\theta^+\cdot\bn^+
+\sigma^-\nabla u_\theta^-\cdot\bn^-\bigr).
\end{equation}
When the spline space is $C^{p-1}$ across a coefficient-homogeneous interface and $p\ge 2$, the flux jump in~\eqref{eq:jump_residual} vanishes. For a Neumann face $B\subset\partial E\cap\Gamma_N$, we define
\begin{equation}
\label{eq:neumann_residual}
\mathcal N_B[u_\theta]
:=g-\sigma\nabla u_\theta\cdot\bn.
\end{equation}

\begin{lemma}[Residual representation]
\label{lem:residual_representation}
For every \(v\in V\),
\begin{equation}
\label{eq:residual_representation}
\langle \mathfrak R_h(\btheta),v\rangle
=
\sum_{E\in\mathcal T_h(\btheta)} 
\int_E \mathcal R_E[u_{\theta}]\,v\,\dx
+
\sum_{I\in\mathcal B_{\mathrm{int}}^\star(\btheta)} 
\int_I \mathcal J_I[u_{\theta}]\,v\,\mathrm{d}S
+
\sum_{B\in\mathcal B_N(\btheta)} 
\int_B \mathcal N_B[u_{\theta}]\,v\,\mathrm{d}S.
\end{equation}
\end{lemma}

\begin{proof}
Fix \(v\in V\). By~\eqref{eq:residual_functional},
\[
\langle\mathfrak R_h(\btheta),v\rangle
=(f,v)_\Omega+\langle g,v\rangle_{\Gamma_N}
-(\sigma\nabla u_\theta,\nabla v)_\Omega
-(\bs\beta\cdot\nabla u_\theta,v)_\Omega
-(\alpha u_\theta,v)_\Omega .
\]
Within each element \(E\in\mathcal T_h(\btheta)\), both \(u_\theta\) and \(\sigma\)
are smooth---the former as the push-forward of a polynomial under the
patchwise-smooth geometry map, the latter because coefficient interfaces lie on
faces of \(\mathcal B_{\rm int}^\star(\btheta)\)---so integrating
\((\sigma\nabla u_\theta,\nabla v)_E\) by parts elementwise and summing yields
\[
\langle\mathfrak R_h(\btheta),v\rangle
=\sum_{E\in\mathcal T_h(\btheta)}\int_E \mathcal R_E[u_\theta]\,v\,\dx
\;-\;\sum_{E\in\mathcal T_h(\btheta)}\int_{\partial E}
(\sigma\nabla u_\theta\cdot\bn_E)\,v\,\mathrm{d}S
\;+\;\langle g,v\rangle_{\Gamma_N},
\]
with \(\bn_E\) being the outward unit normal of \(E\) and \(\mathcal R_E\) as
in~\eqref{eq:element_residual}. We regroup the face sum. Each interior face
\(I=\partial E^+\cap\partial E^-\) is visited twice, with opposite outward normals
\(\bn^+=-\bn^-\), and its two contributions add up to
\(\int_I\mathcal J_I[u_\theta]\,v\,\mathrm{d}S\) by the sign
convention~\eqref{eq:jump_residual}; where the normal flux is continuous
(\(C^{p-1}\) knot lines with \(p\ge2\) and coefficient-homogeneous \(\sigma\)) the
integrand vanishes, so only the faces in \(\mathcal B_{\rm int}^\star(\btheta)\)
remain. Faces on \(\Gamma_D\) contribute nothing because \(v\) has zero trace
there, and on each face \(B\subset\Gamma_N\), the boundary term combines with the
Neumann load into \(\int_B\mathcal N_B[u_\theta]\,v\,\mathrm{d}S\)
by~\eqref{eq:neumann_residual}. Collecting terms gives~\eqref{eq:residual_representation}.
\end{proof}

\paragraph{Residual estimator and loss.}
Let $h_E$ be the diameter of an element, and let $h_I$ and $h_B$ denote the corresponding face sizes. Let
\[
\mu:=\alpha-\frac12\nabla\cdot\bs\beta,
\qquad
\sigma_E:=\operatorname*{ess\,inf}_{E}\sigma,
\qquad
\mu_E:=\operatorname*{ess\,inf}_{E}\mu.
\]
We use
\begin{equation}
\label{eq:rho_weight}
\rho_E:=
\begin{cases}
\min\{h_E/\sigma_E^{1/2},\,\mu_E^{-1/2}\}, & \mu_E>0,\\
h_E/\sigma_E^{1/2}, & \mu_E\le 0.
\end{cases}
\end{equation}
Under Assumption~\ref{ass:coercive_regime} one has \(\mu_E\ge0\), so the second branch is attained only in the boundary case \(\mu_E=0\). For an interior face shared by $E^+$ and $E^-$, set $\sigma_I:=\max(\sigma_{E^+},\sigma_{E^-})$; for a Neumann face $B\subset\partial E$, set $\sigma_B:=\sigma_E$. The estimator combines the three residual contributions from
Lemma~\ref{lem:residual_representation}, with the standard mesh- and
coefficient-dependent weights:
\begin{equation}
\label{eq:robust_estimator}
\eta^2(\btheta)
:=
\sum_{E\in\mathcal T_h(\btheta)}\rho_E^2
\norm{\mathcal R_E[u_\theta]}_{L^2(E)}^2
+
\sum_{I\in\mathcal B_{\rm int}^\star(\btheta)}
\frac{h_I}{\sigma_I}
\norm{\mathcal J_I[u_\theta]}_{L^2(I)}^2
+
\sum_{B\in\mathcal B_N(\btheta)}
\frac{h_B}{\sigma_B}
\norm{\mathcal N_B[u_\theta]}_{L^2(B)}^2.
\end{equation}
The optimization loss is
\begin{equation}
\label{eq:loss_estimator}
\mathcal L(\btheta):=\frac12\eta^2(\btheta).
\end{equation}
In one dimension, the faces are knots, so the face norms in~\eqref{eq:robust_estimator} reduce to point evaluations of the corresponding jumps, with the face weights taken as \(h_I:=\min(h_{E^+},h_{E^-})\) and \(h_B:=h_E\). We next state the conditions under which this estimator controls the energy error.

\begin{assumption}[Coercive regime]
\label{ass:coercive_regime}
We assume \(\Gammain\subseteq\Gamma_D\) and
\[
\mu := \alpha - \tfrac12 \nabla\cdot \bs\beta \ge 0
\qquad \text{a.e.\ in }\Omega.
\]
If \(\mu\equiv 0\), we further assume that \(\Gamma_D\) has 
positive measure, so that the energy 
seminorm~\eqref{eq:energy_norm} is a norm on~\(V\).
\end{assumption}

\begin{assumption}[Admissible mesh family]
\label{ass:admissible_mesh}
Let \(\Theta_{\mathrm{ad}}\subset\R^m\) denote the set of 
admissible mesh parameters. For every 
\(\btheta\in\Theta_{\mathrm{ad}}\):
\begin{enumerate}
\item all element sizes satisfy 
  \(h_E(\btheta)\ge h_{\min}>0\);
\item the mesh family 
  \(\{\mathcal T_h(\btheta)\}_{\btheta\in\Theta_{\mathrm{ad}}}\) 
  is uniformly shape-regular with constant 
  \(\gamma_{\mathrm{sh}}\) (in one dimension this condition is 
  automatically satisfied);
\item the geometry map~\eqref{eq:geometry_map} is bi-Lipschitz on
  each patch interior with constants independent
  of~\(\btheta\);
\item the mesh family is locally quasi-uniform: there exists
  \(\gamma_{\mathrm{loc}}\ge1\), independent of \(\btheta\), such that
  \(h_E\le\gamma_{\mathrm{loc}}\,h_{E'}\) whenever \(E\) and \(E'\)
  intersect the support of a common basis function.
\end{enumerate}
\end{assumption}

\noindent Under Assumption~\ref{ass:coercive_regime}, the bilinear 
form \(B\) is coercive on \(V\) with respect to the energy norm
\begin{equation}
\label{eq:energy_norm}
\|v\|_{\mathcal E}:=\left( 
\|\sigma^{1/2}\nabla v\|_{L^2(\Omega)}^2 
+ \|\mu^{1/2}v\|_{L^2(\Omega)}^2 
\right)^{1/2}.
\end{equation}
Indeed, for every \(v\in V\),
\begin{equation}
\label{eq:coercivity_identity}
B(v,v)
=
\|\sigma^{1/2}\nabla v\|_{L^2(\Omega)}^2
+
\|\mu^{1/2}v\|_{L^2(\Omega)}^2
+
\frac12 \int_{\Gammaout} 
(\bs\beta\cdot \bn)\,v^2\,\mathrm d\Gamma
\ge \|v\|_{\mathcal E}^2,
\end{equation}
because \(v=0\) on \(\Gamma_D\supseteq \Gammain\). In particular, the coercivity constant is exactly \(1\). We equip the
dual space \(V'\) with the norm induced
by~\(\|\cdot\|_{\mathcal E}\):
\[
\|\ell\|_{V'}:=\sup_{v\in V\setminus\{0\}}
\frac{\langle \ell,v\rangle}{\|v\|_{\mathcal E}}.
\]

\begin{proposition}[Reliability]
\label{prop:dual_reliability}
Under Assumptions~\ref{ass:coercive_regime} and~\ref{ass:admissible_mesh}, and assuming the discrete Dirichlet data are imposed exactly (\(u_{D,h}=u_D\) on \(\Gamma_D\), so that \(u-u_\theta\in V\)), there exists $C_{\rm rel}>0$, independent of $\btheta\in\Theta_{\rm ad}$, such that
\begin{equation}
\label{eq:global_reliability}
\|u-u_\theta\|_{\mathcal E}
\le \|\mathfrak R_h(\btheta)\|_{V'}
\le C_{\rm rel}\eta(\btheta).
\end{equation}
The constant depends only on the spline degree and continuity, the shape-regularity and geometry constants, and the bounds on the PDE coefficients. If the Dirichlet data are imposed only approximately, the bound carries an additional data-oscillation term measuring \(\|u_D-u_{D,h}\|\) on \(\Gamma_D\).
\end{proposition}

\begin{proof}
Let \(e:=u-u_\theta\). Since the Dirichlet data are imposed exactly, \(e\in V\).
Using the continuous problem~\eqref{eq:weak_form} and the
definition~\eqref{eq:residual_functional} of the residual,
\[
B(e,v)=B(u,v)-B(u_\theta,v)=F(v)-B(u_\theta,v)
=\langle\mathfrak R_h(\btheta),v\rangle
\qquad \forall v\in V .
\]
Taking \(v=e\) and using the coercivity identity~\eqref{eq:coercivity_identity},
\[
\|e\|_{\mathcal E}^2
\le B(e,e)
=\langle\mathfrak R_h(\btheta),e\rangle
\le \|\mathfrak R_h(\btheta)\|_{V'}\,\|e\|_{\mathcal E},
\]
hence \(\|u-u_\theta\|_{\mathcal E}\le\|\mathfrak R_h(\btheta)\|_{V'}\), which is the
first inequality in~\eqref{eq:global_reliability}.

It remains to bound the dual norm of the residual by the computable estimator. Let
\(v\in V\) be arbitrary and let \(I_hv\in V_h^0(\btheta)\) be a Cl\'ement- or
Scott--Zhang-type quasi-interpolant adapted to spline spaces and preserving the
homogeneous Dirichlet condition
\cite{buffa2022mathematical,clement1975approximation,scott1990finite}. On
admissible meshes---using, in particular, the local quasi-uniformity of item~4
of Assumption~\ref{ass:admissible_mesh}---it satisfies, for every element \(E\) and every face \(F\),
\[
\|v-I_hv\|_{L^2(E)}\le C\rho_E\,\|v\|_{\mathcal E(\omega_E)},
\qquad
\|v-I_hv\|_{L^2(F)}\le C\Bigl(\frac{h_F}{\sigma_F}\Bigr)^{1/2}\|v\|_{\mathcal E(\omega_F)},
\]
where the patches \(\omega_E\), \(\omega_F\) have finite overlap; the unweighted form of these estimates is classical for spline spaces \cite{buffa2022mathematical}, and the coefficient- and reaction-weighted form follows by combining them patchwise with the scaling arguments of \cite{verfurth2005reactiondiffusion,verfurth2005convectiondiffusion}. By Galerkin
orthogonality, \(\langle\mathfrak R_h(\btheta),I_hv\rangle=0\), and hence
\(\langle\mathfrak R_h(\btheta),v\rangle=\langle\mathfrak R_h(\btheta),w\rangle\)
with \(w:=v-I_hv\). The residual representation of
Lemma~\ref{lem:residual_representation} then gives
\[
\langle\mathfrak R_h(\btheta),v\rangle
=\sum_{E\in\mathcal T_h(\btheta)}\int_E\mathcal R_E[u_\theta]\,w\,\dx
+\sum_{I\in\mathcal B_{\rm int}^\star(\btheta)}\int_I\mathcal J_I[u_\theta]\,w\,\mathrm{d}S
+\sum_{B\in\mathcal B_N(\btheta)}\int_B\mathcal N_B[u_\theta]\,w\,\mathrm{d}S .
\]
We estimate the three sums separately. For the element residuals, the
Cauchy--Schwarz inequality, the first interpolation estimate, and the finite
overlap of the patches \(\omega_E\) give
\begin{eqnarray*}
\Bigl|\sum_{E\in\mathcal T_h(\btheta)}\int_E\mathcal R_E[u_\theta]\,w\,\dx\Bigr|
&\le&
\Bigl(\sum_{E}\rho_E^2\|\mathcal R_E[u_\theta]\|_{L^2(E)}^2\Bigr)^{1/2}
\Bigl(\sum_{E}\rho_E^{-2}\|w\|_{L^2(E)}^2\Bigr)^{1/2} \\
&\le&
C\Bigl(\sum_{E}\rho_E^2\|\mathcal R_E[u_\theta]\|_{L^2(E)}^2\Bigr)^{1/2}\|v\|_{\mathcal E}.
\end{eqnarray*}
The trace interpolation estimate with \(F=I\) yields, in the same way,
\[
\Bigl|\sum_{I\in\mathcal B_{\rm int}^\star(\btheta)}\int_I\mathcal J_I[u_\theta]\,w\,\mathrm{d}S\Bigr|
\le
C\Bigl(\sum_{I\in\mathcal B_{\rm int}^\star(\btheta)}
\frac{h_I}{\sigma_I}\|\mathcal J_I[u_\theta]\|_{L^2(I)}^2\Bigr)^{1/2}\|v\|_{\mathcal E},
\]
and the same argument on Neumann faces, with \(F=B\), gives
\[
\Bigl|\sum_{B\in\mathcal B_N(\btheta)}\int_B\mathcal N_B[u_\theta]\,w\,\mathrm{d}S\Bigr|
\le
C\Bigl(\sum_{B\in\mathcal B_N(\btheta)}
\frac{h_B}{\sigma_B}\|\mathcal N_B[u_\theta]\|_{L^2(B)}^2\Bigr)^{1/2}\|v\|_{\mathcal E}.
\]
Combining the three bounds and recalling the definition~\eqref{eq:robust_estimator}
of \(\eta(\btheta)\), we obtain
\[
|\langle\mathfrak R_h(\btheta),v\rangle|
\le C_{\rm rel}\,\eta(\btheta)\,\|v\|_{\mathcal E}.
\]
The constant \(C_{\rm rel}\) depends only on the spline degree and continuity, the
uniform shape-regularity and geometry constants, and the coefficient bounds; in
particular, it is independent of \(\btheta\in\Theta_{\rm ad}\). Taking the supremum
over \(v\in V\setminus\{0\}\) yields
\(\|\mathfrak R_h(\btheta)\|_{V'}\le C_{\rm rel}\,\eta(\btheta)\), the second
inequality in~\eqref{eq:global_reliability}.
\end{proof}
\begin{remark}[Anisotropic meshes]
\label{rem:anisotropy}
The minimum-size constraint prevents element collapse but does not control aspect ratios in tensor-product meshes. Therefore Assumption~\ref{ass:admissible_mesh} is a condition for the theory, not an automatic consequence of the parametrization. In practice, boundary layers may produce anisotropic elements. A fully anisotropic estimator would require directional weights and corresponding trace estimates \cite{formaggia2001new,kunert2001robust,apel2011anisotropic}.
\end{remark}

\begin{theorem}[Local efficiency]
\label{thm:local_efficiency}
Under Assumptions~\ref{ass:coercive_regime} and~\ref{ass:admissible_mesh}, the local contribution $\eta_E(\btheta)$ of~\eqref{eq:robust_estimator} satisfies
\begin{equation}
\label{eq:local_efficiency}
\eta_E(\btheta)
\le C_{\rm eff}
\left(
\|u-u_\theta\|_{{\mathcal E},\omega_E}
+\operatorname{osc}_{\omega_E}(\btheta)
\right),
\end{equation}
where $\omega_E$ is the patch of elements sharing a vertex with $E$. The constant is independent of $\btheta$ and of the local mesh sizes; it depends on the polynomial degree and continuity, the shape-regularity and geometry constants, and the coefficient bounds---through the convective term, in particular, on $\|\bs\beta\|_{L^\infty}\operatorname{diam}(\Omega)/\sigma_{\min}$, the mechanism behind the convection-dominated deterioration of Remark~\ref{rem:coefficient_regimes}. The term $\operatorname{osc}_{\omega_E}$ contains the usual data, coefficient, and residual-approximation oscillations.
\end{theorem}

\begin{proof}
We consider a standard local bubble-function argument~\cite{verfurth2013posteriori,ainsworth2000posteriori}, adapted to the residual
components in~\eqref{eq:robust_estimator}. Let \(e:=u-u_\theta\). We first make
the local contribution of the estimator explicit. Let \(\mathcal F_E^\star\)
denote the set of interior faces in \(\mathcal B_{\rm int}^\star(\btheta)\)
contained in \(\partial E\), and \(\mathcal B_E^N\) the set of Neumann faces
contained in \(\partial E\). Up to harmless sharing factors on interior faces,
we may take
\[
\eta_E^2(\btheta)
:=
\rho_E^2\|\mathcal R_E[u_\theta]\|_{L^2(E)}^2
+\sum_{I\in\mathcal F_E^\star}
\frac{h_I}{\sigma_I}\|\mathcal J_I[u_\theta]\|_{L^2(I)}^2
+\sum_{B\in\mathcal B_E^N}
\frac{h_B}{\sigma_B}\|\mathcal N_B[u_\theta]\|_{L^2(B)}^2 .
\]
Let \(\Pi_E\), \(\Pi_I\), and \(\Pi_B\) be local polynomial projections on
elements, interior faces, and Neumann faces, respectively; the differences
between the residuals and their projections are collected in the oscillation
term \(\operatorname{osc}_{\omega_E}(\btheta)\), which contains the data,
coefficient, and residual-approximation oscillations on the patch \(\omega_E\).

We first estimate the element residual. Let \(b_E\) be the standard element
bubble on \(E\) and set \(v_E:=b_E\,\Pi_E\mathcal R_E[u_\theta]\), extended by
zero outside \(E\). Since \(v_E\in V\) vanishes on \(\partial E\),
Lemma~\ref{lem:residual_representation} gives
\[
\int_E \mathcal R_E[u_\theta]\,v_E\,\dx
=\langle \mathfrak R_h(\btheta),v_E\rangle .
\]
Using \(B(e,v_E)=\langle\mathfrak R_h(\btheta),v_E\rangle\), the boundedness of
\(B\), and the inverse and scaling estimates for element bubbles~\cite{verfurth2013posteriori} (in the reaction-weighted case, the modified cut-off bubbles of~\cite{verfurth2005reactiondiffusion}, which yield the weight \(\rho_E\) robustly), we obtain
\[
\rho_E\,\|\Pi_E\mathcal R_E[u_\theta]\|_{L^2(E)}
\le C\,\|e\|_{\mathcal E,E}.
\]
Adding and subtracting \(\Pi_E\mathcal R_E[u_\theta]\) yields
\[
\rho_E\,\|\mathcal R_E[u_\theta]\|_{L^2(E)}
\le C\bigl(
\|e\|_{\mathcal E,E}
+\rho_E\,\|\mathcal R_E[u_\theta]-\Pi_E\mathcal R_E[u_\theta]\|_{L^2(E)}
\bigr),
\]
and the second term is part of the volume oscillation.

We now consider an interior face
\(I=\partial E^+\cap\partial E^-\in\mathcal B_{\rm int}^\star(\btheta)\). Let
\(b_I\) be a face bubble supported on the two-element patch
\(\omega_I:=E^+\cup E^-\), and let \(v_I\) be a lifting of
\(b_I\,\Pi_I\mathcal J_I[u_\theta]\) from \(I\) to \(\omega_I\)---the polynomial face extension composed with the face bubble~\cite{verfurth2013posteriori}---extended by zero outside \(\omega_I\). Applying Lemma~\ref{lem:residual_representation}
with \(v_I\) and using \(B(e,v_I)=\langle\mathfrak R_h(\btheta),v_I\rangle\),
\[
\int_I \mathcal J_I[u_\theta]\,v_I\,\mathrm{d}S
=B(e,v_I)
-\sum_{E'\subset\omega_I}\int_{E'}\mathcal R_{E'}[u_\theta]\,v_I\,\dx .
\]
The trace and inverse estimates for face bubbles, together with the
element-residual bound already obtained, give
\[
\Bigl(\frac{h_I}{\sigma_I}\Bigr)^{1/2}\|\mathcal J_I[u_\theta]\|_{L^2(I)}
\le C\bigl(\|e\|_{\mathcal E,\omega_I}+\operatorname{osc}_{\omega_I}(\btheta)\bigr).
\]
The Neumann term is estimated in the same way: for
\(B\subset\partial E\cap\Gamma_N\), a boundary-face bubble \(b_B\) and a lifting
\(v_B\) of \(b_B\,\Pi_B\mathcal N_B[u_\theta]\), supported on the element patch
adjacent to \(B\), yield, by Lemma~\ref{lem:residual_representation} and the
corresponding trace and inverse estimates,
\[
\Bigl(\frac{h_B}{\sigma_B}\Bigr)^{1/2}\|\mathcal N_B[u_\theta]\|_{L^2(B)}
\le C\bigl(\|e\|_{\mathcal E,\omega_B}+\operatorname{osc}_{\omega_B}(\btheta)\bigr).
\]
Combining the bounds for the element residual, the interior flux jumps, and the
Neumann residuals, and using the finite overlap of the element and face patches,
gives
\[
\eta_E(\btheta)
\le C_{\rm eff}\bigl(
\|u-u_\theta\|_{\mathcal E,\omega_E}
+\operatorname{osc}_{\omega_E}(\btheta)
\bigr).
\]
The constant \(C_{\rm eff}\) depends only on the polynomial degree and
continuity, the coefficient bounds, and the uniform shape-regularity and
geometry constants of the admissible mesh family; in particular, it is
independent of \(\btheta\) and of the local mesh sizes. This
proves~\eqref{eq:local_efficiency}.
\end{proof}

\begin{remark}[Coefficient regimes]
\label{rem:coefficient_regimes}
For the Poisson case \(\sigma\equiv1\), \(\bs\beta\equiv0\), and
\(\alpha\equiv0\), the energy norm is the \(H^1\)-seminorm and
Proposition~\ref{prop:dual_reliability} and Theorem~\ref{thm:local_efficiency}
give the standard two-sided residual bounds, up to oscillation. For
advection--diffusion--reaction problems satisfying
Assumption~\ref{ass:coercive_regime}, the same residual estimator is used with
the corresponding energy norm; in convection-dominated regimes, however, the
constants may deteriorate with the perturbation parameter
\cite{petzoldt2002aposteriori,verfurth2005convectiondiffusion,sangalli2008robust,apel2011anisotropic,verfurth2005reactiondiffusion}.
A separate situation arises in the indefinite regime
\(\mu=\alpha-\tfrac12\nabla\cdot\bs\beta<0\), of which the Helmholtz operator
\(-\Delta u-c^{2}u\) is the canonical example. There,
Assumption~\ref{ass:coercive_regime} is violated, the bilinear form is no longer
coercive on \(V\), and the reliability constant in
Proposition~\ref{prop:dual_reliability} cannot be guaranteed: for wavenumbers
close to a discrete resonance, the inf--sup constant degenerates and the residual
no longer bounds the energy error in the sense of~\eqref{eq:global_reliability}
\cite{ihlenburg1995finite,ihlenburg1998finite,babuska1997pollution}. Standard
a posteriori control in this regime requires either a sufficiently resolved mesh
(the asymptotic, pollution-free range \(kh\lesssim 1\)) or wavenumber-explicit
duality arguments \cite{doerfler2013helmholtz}.
\end{remark}

\begin{remark}[Scope of the theory in the numerical experiments]
\label{rem:experiment_scope}
Several experiments of Section~\ref{sec:numerical_examples} fall outside the
assumptions of Proposition~\ref{prop:dual_reliability} and
Theorem~\ref{thm:local_efficiency}, and are interpreted accordingly. First, the
convection-dominated boundary layer of
Section~\ref{subsec:exp5_advection_diffusion} is read as a fixed-DOF mesh-quality
test, not as a robustness result uniform in the diffusion parameter
(cf.~Remark~\ref{rem:coefficient_regimes}). Second, for the indefinite Helmholtz
problem of Section~\ref{subsec:exp_helmholtz}, we claim no reliability guarantee:
the estimator~\eqref{eq:robust_estimator} is used purely as a differentiable
mesh-quality functional, the reported effectivity index is a numerical diagnostic
rather than the realization of a proven bound, and the experiment probes whether
the residual-driven predictor remains useful when the governing operator leaves the
coercive setting. Third, for a piecewise-constant \(\sigma\) with large contrast
\(\sigma_{\max}/\sigma_{\min}\) (Section~\ref{subsec:exp4_lshape}), the
contrast-robustness of the constants requires a quasi-monotonicity condition on
the coefficient distribution around each vertex \cite{petzoldt2002aposteriori};
this is implicit in the interface weights
\(\sigma_I=\max(\sigma_{E^+},\sigma_{E^-})\), and, in its absence, the constants carry
an explicit \(\sigma_{\max}/\sigma_{\min}\) dependence. Finally, the a posteriori
theory presumes a conforming discretization \(V_h\subset V\) with exact Dirichlet
data; the immersed, fictitious-domain construction of
Section~\ref{subsec:exp4_lshape} enforces the re-entrant boundary only
approximately through the \(C^0\) cut, so
Proposition~\ref{prop:dual_reliability} and Theorem~\ref{thm:local_efficiency}
do not rigorously cover that experiment, and the estimator is again used as a
mesh-quality functional rather than as a strict bound.
\end{remark}

\subsection{Gradient computation via the discrete adjoint}
\label{subsec:diff_solve}

The reduced loss is
\[
\widehat{\mathcal L}(\btheta)
:=\mathcal L(\mathbf U(\btheta),\btheta),
\]
where \(\mathbf U(\btheta)\) solves~\eqref{eq:linear_system}. The dependence on \(\btheta\) is explicit through the basis functions, quadrature points, element sizes, and residual weights, and implicit through the solution vector. Differentiating through the internal operations of a direct or iterative linear solver is unnecessary; rather we differentiate the discrete solution equation. By the chain rule,
\begin{equation}
\label{eq:chain_rule}
\nabla_{\btheta}\widehat{\mathcal L}
=
\left(\frac{\partial\mathbf U}{\partial\btheta}\right)^T
\nabla_{\mathbf U}\mathcal L
+
\frac{\partial\mathcal L}{\partial\btheta},
\end{equation}
and differentiating \(\mathbf K(\btheta)\mathbf U(\btheta)=\mathbf F(\btheta)\) yields the solution sensitivity equation
\begin{equation}
\label{eq:solution_sensitivity}
\mathbf K(\btheta)\,\frac{\partial\mathbf U}{\partial\btheta}
=
\frac{\partial\mathbf F}{\partial\btheta}
-
\frac{\partial\mathbf K}{\partial\btheta}\,\mathbf U(\btheta).
\end{equation}
The dense sensitivity \(\partial\mathbf U/\partial\btheta\) is avoided by an implicit discrete adjoint: defining \(\bs\lambda(\btheta)\) as the solution of
\begin{equation}
\label{eq:adjoint_system}
\mathbf K(\btheta)^T\bs\lambda(\btheta)
=\nabla_{\mathbf U}\mathcal L(\mathbf U(\btheta),\btheta),
\end{equation}
combining~\eqref{eq:chain_rule}--\eqref{eq:adjoint_system} gives the reduced gradient
\begin{equation}
\label{eq:adjoint_grad_prop}
\nabla_{\btheta}\widehat{\mathcal L}(\btheta)
=
\bs\lambda(\btheta)^T
\left(
\frac{\partial\mathbf F}{\partial\btheta}(\btheta)
-
\frac{\partial\mathbf K}{\partial\btheta}(\btheta)\,\mathbf U(\btheta)
\right)
+
\frac{\partial\mathcal L}{\partial\btheta}(\mathbf U(\btheta),\btheta).
\end{equation}
This is the discrete counterpart of adjoint calculus for PDE-constrained optimization~\cite{lions1971optimal,hinze2008optimization,blondel2021efficient}. It requires one additional solve with \(\mathbf K(\btheta)^T\); when a factorization of the forward matrix is available, it can be reused, so the marginal cost is small. The dense sensitivity matrix is never formed, and the memory cost is independent of the internal depth of the linear solver.

\paragraph{Realization through reverse-mode AD.}
Modern automatic-differentiation (AD) frameworks~\cite{baydin2018ad,jax2018github} realize the implicit adjoint~\eqref{eq:adjoint_system}--\eqref{eq:adjoint_grad_prop} natively. When the linear solve \(\mathbf U=\mathbf K^{-1}\mathbf F\) is invoked through a differentiable primitive, reverse-mode AD inserts the adjoint solve automatically. This is the implicit-function-theorem pullback of the solve: the adjoint of the solution is \(\bar{\mathbf F}=\mathbf K^{-\top}\bar{\mathbf U}\), which is exactly the multiplier \(\bs\lambda\) of~\eqref{eq:adjoint_system}, and the sensitivity with respect to the matrix is the outer product \(\bar{\mathbf K}=-\bar{\mathbf F}\,\mathbf U^\top\). Crucially, this is not differentiation through the solver iterations, which would give an incorrect gradient before convergence; an iterative solver must instead expose the implicit adjoint explicitly, for example through a custom linear-solve pullback. The remaining derivatives of \(\mathbf K\), \(\mathbf F\), and \(\mathcal L\) reduce to vector--Jacobian products on element-level assembly that AD evaluates efficiently. In our JAX implementation, the linear solve exposes this implicit-function-theorem pullback either natively (\textit{jnp.linalg.solve}) or through an explicit VJP registered on the solve that reuses the forward factorization---the two are verified to coincide to finite-difference accuracy---so gradients are obtained by direct \textit{jax.value\_and\_grad} on the reduced loss, without deriving an adjoint PDE by hand and without differentiating through any iterative loop. Algorithm~\ref{alg:radapt-iga} summarizes the resulting per-instance optimization for a single realization of the PDE data. Its continuation ladder uses a level-transfer operator \(\mathcal P_{\ell-1}^{\ell}\), which initializes level \(\ell\) by resampling the optimized mesh of level \(\ell-1\) on the finer \(N_\ell\)-element partition, preserving the learned grading. This optimized mesh is the target that the parametric method of Section~\ref{sec:parametric_inn} learns to predict: rather than re-solving the optimization for every parameter value, a single network is trained to reproduce the optimized grading across the whole family in one forward pass. Accordingly, Algorithm~\ref{alg:radapt-iga} is not evaluated on its own; all reported experiments use the parametric Algorithm~\ref{alg:parametric-radapt}.

\begin{algorithm}[htbp]
\caption{Residual-driven $r$-adaptive IGA}
\label{alg:radapt-iga}
\begin{algorithmic}[1]
\Require Degree $p$, continuity $c$; increasing element counts
$N_0<N_1<\cdots<N_L$ (per direction) for the coarse-to-fine levels,
one knot vector per level; minimum element
size $h_{\min}$; PDE data and quadrature rules; optimizer settings
$(K_{\max},\{\gamma_k\})$; stopping tolerances $(\tau_{\rm grad},\tau_{\rm rel})$.
\State Initialize $\bs\theta_0^0$ as the uniform mesh with $N_0$ elements.
\For{$\ell=0,1,\ldots,L$}
    \If{$\ell>0$}
        \State $\bs\theta_\ell^0\gets\mathcal P_{\ell-1}^{\ell}(\bs\theta_{\ell-1}^\ast)$
        \Comment{warm start: resample the optimized mesh map of level $\ell-1$ on $N_\ell$ elements}
    \EndIf
    \State $J_\ell^\ast\gets+\infty$
    \For{$k=0,1,\ldots,K_{\max}-1$}
        \State Build $V_h(\bs\theta_\ell^k)$ and solve $\mathbf K(\bs\theta_\ell^k)\mathbf U_\ell^k=\mathbf F(\bs\theta_\ell^k)$.
        \State Evaluate $\mathcal L_\ell^k=\mathcal L(\mathbf U_\ell^k,\bs\theta_\ell^k)$.
        \State Compute $\mathbf g_\ell^k=\nabla_{\bs\theta}\widehat{\mathcal L}_\ell(\bs\theta_\ell^k)$ by reverse-mode AD, i.e., \eqref{eq:adjoint_grad_prop}.
        \State Update $\bs\theta_\ell^{k+1}\gets\operatorname{Adam}(\bs\theta_\ell^k,\mathbf g_\ell^k;\gamma_k)$.
        \If{$\mathcal L_\ell^k<J_\ell^\ast$}
            \State $(J_\ell^\ast,\bs\theta_\ell^\ast,\mathbf U_\ell^\ast)\gets(\mathcal L_\ell^k,\bs\theta_\ell^k,\mathbf U_\ell^k)$
        \EndIf
        \If{$\|\mathbf g_\ell^k\|\le\tau_{\rm grad}$ \textbf{or} $\bigl(k\ge1$ \textbf{and} $|\mathcal L_\ell^k-\mathcal L_\ell^{k-1}|\le\tau_{\rm rel}\,|\mathcal L_\ell^{k-1}|\bigr)$}
            \State \textbf{break}
        \EndIf
    \EndFor
    \State Export the optimized mesh and solution $(\bs\theta_\ell^\ast,\mathbf U_\ell^\ast)$.
\EndFor
\end{algorithmic}
\end{algorithm}

\section{Parametric residual-informed neural mesh}
\label{sec:parametric_inn}

We now turn to families of problems. The coefficients and data
of~\eqref{eq:general_pde} depend on a parameter \(\nu\in\mathcal P\subset\mathbb R^{d_\nu}\), with \(\mathcal P\) compact; the geometry and mesh topology stay fixed. For an admissible mesh \(\btheta\), let \(u_{h,\btheta,\nu}\in V_h(\btheta)\) be the Galerkin solution for parameter \(\nu\), and \(\eta(\btheta;\nu)\) its residual estimator~\eqref{eq:robust_estimator}. The central idea is the following: a neural network predicts a mesh, not a solution. We call this construction the residual-informed neural mesh. Given \(\nu\), it returns an admissible knot configuration; the solution is given by a standard Galerkin solve. The network never approximates \(u_{h,\btheta,\nu}\). This is what separates the method from neural PDE solvers with moving meshes~\cite{omella2024radapt}, where the solution itself is the neural ansatz, and from Ritz-based parametric $r$-adaptivity~\cite{aballay2025radaptivefiniteelementmethod}, which is limited to symmetric coercive problems. Residual losses have been used to train neural surrogates of the solution~\cite{taylor2023deepfourier,uriarte2023doubleritz,rojas2024rvpinn}; here, we move that role onto the mesh and keep the solution a Galerkin solution. In the coercive case, the loss inherits the reliability of Section~\ref{sec:proposed_approach}. Figure~\ref{fig:parametric_inn_flowchart} summarizes the offline and online stages of the residual-informed neural mesh.

\begin{figure}[t]
\centering
\includegraphics[width=\linewidth]{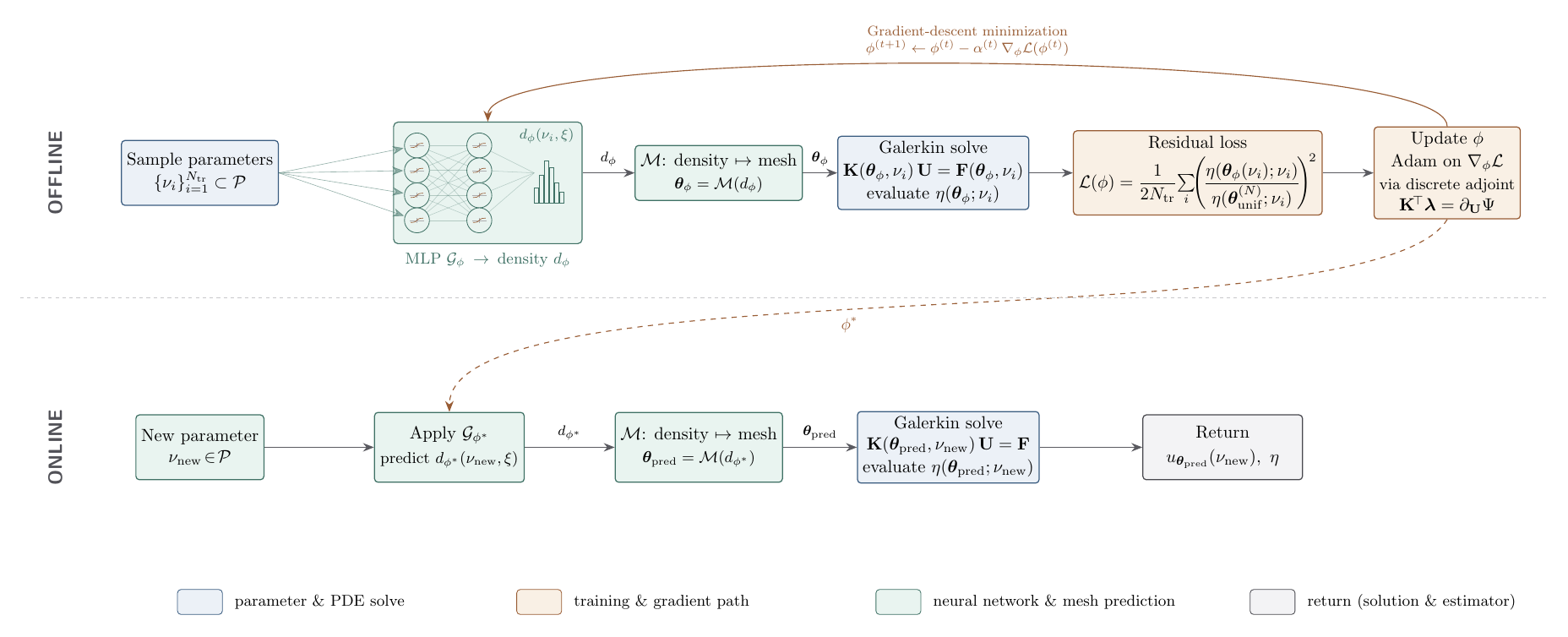}
\caption{Residual-informed neural mesh flowchart. \textbf{Offline}: the network weights \(\phi\) are trained by reverse-mode AD gradients on the residual loss, with one Galerkin solve per training sample. \textbf{Online}: the trained network predicts a per-element mesh density \(d_\phi\) in a single forward pass, the reconstruction \(\mathcal M\) turns it into an admissible mesh \(\btheta_\phi\), and one Galerkin solve produces the solution and the residual estimator.}
\label{fig:parametric_inn_flowchart}
\end{figure}

\subsection{Positional density network and mesh predictor}
\label{subsec:inn_mesh_predictor}
A direct map \(\mathcal P\to\mathbb R^m\) from the parameter to the \(m\) element sizes would be tied to a fixed number of elements: changing the refinement level would change the output dimension and require retraining. We avoid this by predicting a \emph{mesh density} instead of a fixed-size vector: the network returns a scalar density at any single coordinate---high where elements should be small, low where they should be large---which we sample at the element centers and convert into sizes. Because the network reads one coordinate at a time, its architecture is independent of \(m\), so the same trained weights produce a mesh at any refinement level and enable coarse-to-fine continuation (Section~\ref{subsec:inn_offline_protocols}).

The construction is guided by four requirements on the predicted sizes: they must be (i) positive, (ii) sum to the domain length so the mesh tiles it, (iii) bounded away from zero so no element collapses onto its neighbor, and (iv) independent of the element count. A raw network cannot enforce these, so we pass its scalar output through a fixed, parameter-free map in three differentiable steps. We first center the densities, since only their relative values matter and adding a constant should leave the mesh unchanged. We then saturate them into a bounded range through a \(\tanh\): without this bound the following step could make one element exponentially smaller than the rest, and the bound \(T\) sets how aggressively the mesh may be graded. Finally, a softmax function maps the bounded densities to positive fractions summing to one, rescaled to the domain length and floored at \(h_{\min}\), which secures (i)--(iii); requirement (iv) holds because only the number of sample points, not the network, changes with the level. We now make each step precise.

\paragraph{Mesh-density network.}
The density is a coordinate-based network
\begin{equation}
\label{eq:positional_density_definition}
\mathcal G_\phi:\mathcal P\times\widehat\Omega\to\mathbb R,
\qquad
(\nu,\xi)\mapsto \mathcal G_\phi(\nu,\xi),
\end{equation}
a fully connected multilayer perceptron (MLP) with weights and biases \(\phi=\{W_k,\bs b_k\}_{k=1}^{L}\), two hidden layers, and a smooth (\(\tanh\)) activation. It maps a problem parameter \(\nu\) and a reference coordinate \(\xi\in\widehat\Omega\) to a scalar mesh density \(d(\nu,\xi)=\mathcal G_\phi(\nu,\xi)\); where the density is larger, elements will be smaller. To build an \(m\)-element mesh we sample the continuous density \(\mathcal G_\phi(\nu,\cdot)\) once per element, at the element centers
\begin{equation}
\label{eq:cell_centers}
\xi_i^{c}=\frac{i-\tfrac12}{m},\qquad i=1,\ldots,m,
\end{equation}
giving \(m\) values \(d_i(\nu)=\mathcal G_\phi(\nu,\xi_i^{c})\). Each element needs exactly one representative density value, and its center is the natural sampling point; changing the refinement level changes only the number \(m\) of sample points \(\xi_i^{c}\), never the network \(\mathcal G_\phi\), which is a continuous function of the coordinate. The fixed map \(\mathcal M\) turns these \(m\) densities into element sizes through the three steps anticipated above. First, since only relative densities matter, we remove the additive gauge freedom by centering,
\begin{equation}
\label{eq:density_centering}
\tilde d_i(\nu)=d_i(\nu)-\frac1m\sum_{j=1}^m d_j(\nu).
\end{equation}
Second, we bound the densities to a fixed range \([-T,T]\) through
\begin{equation}
\label{eq:density_saturation}
q_i(\nu)=T\tanh\bigl(\tilde d_i(\nu)/T\bigr),
\end{equation}
where \(T>0\) is a fixed saturation amplitude. This bound caps how strongly the mesh can be graded: in the next step the ratio between the largest and smallest element is at most \(e^{2T}\), which prevents the network from collapsing elements onto a single point. Third, a softmax converts the bounded densities into positive sizes that sum to the segment length and respect a minimum element size \(h_{\min}\),
\begin{equation}
\label{eq:density_to_sizes}
h_i(\nu)=h_{\min}+\bigl(L-m\,h_{\min}\bigr)\,\frac{e^{q_i(\nu)}}{\sum\limits_{j=1}^m e^{q_j(\nu)}},
\end{equation}
so that \(\sum_i h_i=L\) and \(h_i\ge h_{\min}\) for every element. The element breakpoints follow by cumulative summation of the \(h_i\), as in Section~\ref{sec:mesh_param}. We write \(\btheta_\phi(\nu)\) for the resulting mesh; the construction is differentiable in \(\phi\), so the loss gradient passes through it. Moreover, since \(q_i\in[-T,T]\), any two element sizes produced by \(\mathcal M\) satisfy \(h_i/h_j\le e^{2T}\), so the predicted meshes fulfill the local quasi-uniformity of item~4 of Assumption~\ref{ass:admissible_mesh} with \(\gamma_{\mathrm{loc}}\le e^{2T}\), uniformly in \(\phi\) and \(\nu\); in the non-parametric setting of Section~\ref{sec:mesh_param}, \(\Theta_{\mathrm{ad}}\) is restricted accordingly.

\paragraph{Interfaces, blocks, and tensor products.}
In \(d\) dimensions, the density is evaluated independently along each parametric direction, and the mesh is the tensor product of the resulting univariate partitions. A fixed interior interface---a location where the data are non-smooth, such as a material interface where the diffusion coefficient jumps---is always carried as a pinned knot of the prescribed multiplicity, so the discretization stays conforming and \(C^0\) there for every \(\btheta\); together with the \(h_{\min}\) floor and the fixed total element count, this keeps the mesh admissible independently of \(\btheta\). Two treatments of such an interface are available. In the \emph{block} treatment the interface is a segment boundary that splits a direction into blocks, each a segment with its own element budget and its own centers \(\xi_i^{c}\), so the per-block element counts are frozen and the interface keeps a fixed index in the knot vector. This is the appropriate choice when that index must stay static, as in the L-shape of Section~\ref{subsec:exp4_lshape}, whose immersed (Dirichlet-masked) trimmed region is bounded by the interface lines \(x=0.5\) and \(y=0.5\): each direction splits into two blocks and the tensor-product mesh carries four spacing vectors in total (Figure~\ref{fig:blocks}). In the \emph{free-split} treatment a single budget spans the whole direction and the interface is inserted at its fixed location and multiplicity rather than acting as a segment boundary, so the number of elements on each side is itself governed by \(\btheta\); this is preferable when the physics rewards moving resolution across the interface, as in the one-dimensional contrast Helmholtz problem of Section~\ref{subsec:exp_helmholtz}, where the short-wavelength layer draws elements across \(x=0.5\) (Figure~\ref{fig:knot_redistribution}). In both treatments the same network \(\mathcal G_\phi\) predicts the density, blocks being distinguished by an extra block-identity input, and anchor pinning, the \(h_{\min}\) floor, the clamped ends and the fixed total count are unchanged. The classical density view of moving meshes and optimal knot placement underlies this construction~\cite{deboor1974good,deboor1974good2,budd2009adaptive,huang2010adaptive}, here realized as a coordinate-based neural field~\cite{sitzmann2020implicit}.

\begin{figure}[htbp]
\centering
\includegraphics[width=0.5\linewidth]{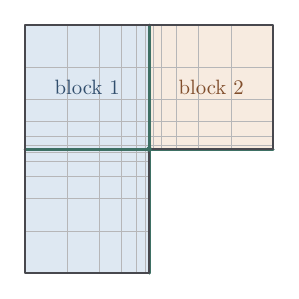}
\caption{Block decomposition on the L-shaped domain. The interface lines at \(x=0.5\) and \(y=0.5\) (green) split each direction into two blocks; block~1 and block~2 label those of the \(x\)-direction. The knot lines are graded toward the re-entrant corner.}
\label{fig:blocks}
\end{figure}

\subsection{Residual loss and uniform reliability}
\label{subsec:inn_loss}

Let \(\{\nu_i\}_{i=1}^n\subset\mathcal P\) be the training set. We train \(\phi\) by minimizing a normalized residual loss,
\begin{equation}
\label{eq:parametric_training_loss}
\mathcal J_{\rm res}(\phi)
=
\frac1{2n}\sum_{i=1}^n
\frac{\eta^2(\btheta_\phi(\nu_i);\nu_i)}
     {\eta^2(\btheta_{\rm unif};\nu_i)+\varepsilon},
\end{equation}
where \(\btheta_{\rm unif}\) is the uniform mesh at the same level and \(\varepsilon>0\) avoids division by zero. Dividing by the uniform-mesh residual makes each term dimensionless and prevents parameters with larger residuals from dominating the training loss. It is not a reference error and needs no precomputed optimal meshes.

\begin{proposition}[Uniform reliability]
\label{prop:parametric_reliability}
Suppose Assumptions~\ref{ass:coercive_regime} and~\ref{ass:admissible_mesh} hold uniformly for \(\nu\in\mathcal P\), and the coefficients and data depend continuously on \(\nu\) in the norms used in the estimator. Then there exists a constant \(\bar C_{\rm rel}<\infty\), independent of \(\nu\) and of the admissible mesh, such that
\begin{equation}
\label{eq:parametric_reliability}
\norm{u_\nu-u_{h,\btheta,\nu}}_{\mathcal E,\nu}
\le
\bar C_{\rm rel}\,\eta(\btheta;\nu),
\qquad
\forall\, \nu\in\mathcal P.
\end{equation}
\end{proposition}

\begin{proof}
For each fixed \(\nu\), Proposition~\ref{prop:dual_reliability} gives a constant \(C_{\rm rel}(\nu)\) depending only on the polynomial degree, continuity, geometry bounds, shape regularity, and coefficient bounds. Three facts make the supremum finite. First, Assumption~\ref{ass:admissible_mesh} holds uniformly in \(\nu\), so the shape-regularity constant \(\gamma_{\rm sh}\) and the geometry constants are \(\nu\)-independent. Second, the coercivity constant equals \(1\) uniformly due to the coercivity identity, so no inf--sup factor enters \(C_{\rm rel}(\nu)\). Third, since \(|\sigma_{\min}(\nu)-\sigma_{\min}(\nu')|\le\|\sigma(\cdot\,;\nu)-\sigma(\cdot\,;\nu')\|_{L^\infty(\Omega)}\), the map \(\nu\mapsto\sigma_{\min}(\nu)\) is continuous; by compactness of \(\mathcal P\), it attains a positive minimum on \(\mathcal P\), bounding the coefficient-dependent part of \(C_{\rm rel}(\nu)\) from above. Hence \(\bar C_{\rm rel}:=\sup_{\nu\in\mathcal P} C_{\rm rel}(\nu)<\infty\).
\end{proof}

\noindent The effectivity index, by contrast, can still vary with \(\nu\) and with the mesh. We report it as a diagnostic and keep the relative energy or \(H^1\) error as the primary accuracy metric.

\subsection{Offline training and online prediction}
\label{subsec:inn_offline_protocols}

\paragraph{Offline training.}
For a fixed sample \(\nu_i\), the reduced objective depends on \(\phi\) through \(\btheta_\phi(\nu_i)\) and through the Galerkin solution \(\mathbf U_i\), defined by
\begin{equation}
\label{eq:parametric_linear_system}
\mathbf K(\btheta_\phi(\nu_i),\nu_i)\,\mathbf U_i
=
\mathbf F(\btheta_\phi(\nu_i),\nu_i).
\end{equation}
Writing the per-sample loss as
\[
\Psi_i(\mathbf U,\btheta)
=
\frac12
\frac{\eta^2(\btheta;\nu_i)}
     {\eta^2(\btheta_{\rm unif};\nu_i)+\varepsilon},
\]
the gradient \(\nabla_\phi\mathcal J_{\rm res}\) is obtained by reverse-mode AD through~\eqref{eq:parametric_linear_system} and the density-to-size map~\eqref{eq:density_to_sizes}, as discussed in Section~\ref{subsec:diff_solve}. This is equivalent to the implicit adjoint solve
\begin{equation}
\label{eq:parametric_adjoint}
\mathbf K(\btheta_\phi(\nu_i),\nu_i)^T\bs\lambda_i
=
\partial_{\mathbf U}\Psi_i(\mathbf U_i,\btheta_\phi(\nu_i))^T,
\end{equation}
followed by the chain rule through \(\btheta_\phi\). Averaging over a mini-batch gives a stochastic gradient of \(\mathcal J_{\rm res}\), which we minimize with Adam~\cite{kingma2015adam}. We use two initializations. Direct: start from \(\mathcal G_\phi(\nu)\approx\text{const}\), so every predicted mesh begins uniform. Coarse-to-fine continuation: train the network at a coarse level, then use the trained weights to initialize training at finer levels. Continuation relies on the level-independence of~\eqref{eq:positional_density_definition} and stabilizes training when the solution is strongly singular; each experiment states which one it uses.

\paragraph{Online prediction.}
For a new \(\nu\notin\{\nu_i\}\), the trained residual-informed neural mesh gives
\begin{equation}
\label{eq:online_prediction}
\btheta_{\rm pred}=\btheta_{\phi^\star}(\nu)
\end{equation}
in one forward pass through \(\mathcal G_{\phi^\star}\), the centering~\eqref{eq:density_centering}, the saturation~\eqref{eq:density_saturation}, and the reconstruction \(\mathcal M\). One Galerkin solve on \(\btheta_{\rm pred}\),
\[
\mathbf K(\btheta_{\rm pred},\nu)\,\mathbf U_{\rm pred}
=\mathbf F(\btheta_{\rm pred},\nu),
\]
then gives the solution, with \(\eta(\btheta_{\rm pred};\nu)\) as a quality check. There is no online optimization: a new \(\nu\) costs one forward pass plus one Galerkin solve.

\begin{algorithm}[htbp]
\caption{Parametric residual-informed neural mesh}
\label{alg:parametric-radapt}
\begin{algorithmic}[1]
\Statex \textbf{Offline training}
\Require Degree $p$, continuity $c$; an increasing sequence of element counts
$N_0<\cdots<N_L$ (per direction) for the coarse-to-fine levels; minimum element size $h_{\min}$; grading cap
$T$; training and validation samples
$\{\nu_i\}_{\rm train},\{\nu_j\}_{\rm val}\subset\mathcal P$; density network
$\mathcal G_\phi$; optimizer settings $(K_{\max},\{\gamma_k\})$; validation
patience $K_{\rm pat}$.
\State Initialize $\phi$ so that $\mathcal G_\phi(\nu)\approx\text{const}$ (uniform meshes); set $J^\star\gets+\infty$.
\For{$\ell=0,1,\ldots,L$} \Comment{coarse-to-fine; the same weights apply at every level via the collocation~\eqref{eq:positional_density_definition}}
    \For{$k=0,1,\ldots,K_{\max}-1$}
        \State Sample a mini-batch $\mathcal B\subseteq\{\nu_i\}_{\rm train}$.
        \For{$\nu_i\in\mathcal B$}
            \State Predict $\btheta_{\phi}(\nu_i)$: forward pass $\mathcal G_\phi$, gauge fix~\eqref{eq:density_centering}, saturation~\eqref{eq:density_saturation}, map $\mathcal M$ at level $N_\ell$.
            \State Solve $\mathbf K(\btheta_\phi(\nu_i),\nu_i)\,\mathbf U_i=\mathbf F(\btheta_\phi(\nu_i),\nu_i)$ and evaluate the residual loss $\Psi_i$.
        \EndFor
        \State Form $\mathcal J_{\rm res}(\phi)$ over $\mathcal B$ (normalized as in~\eqref{eq:parametric_training_loss}) and compute $\nabla_\phi\mathcal J_{\rm res}$ by reverse-mode AD through the solution equation~\eqref{eq:parametric_adjoint} and the density-to-size map.
        \State Update $\phi\gets\operatorname{Adam}(\phi,\nabla_\phi\mathcal J_{\rm res};\gamma_k)$.
        \State Evaluate $\mathcal J_{\rm res}^{\rm val}$ on $\{\nu_j\}_{\rm val}$; if $\mathcal J_{\rm res}^{\rm val}<J^\star$, set $(J^\star,\phi^\star)\gets(\mathcal J_{\rm res}^{\rm val},\phi)$.
        \If{$\mathcal J_{\rm res}^{\rm val}$ has not improved for $K_{\rm pat}$ consecutive iterations}
            \State \textbf{break} \Comment{per-level early stopping on held-out $\nu$}
        \EndIf
    \EndFor
\EndFor
\State \Return trained weights $\phi^\star$.
\Statex \textbf{Online prediction} (new $\nu\notin\{\nu_i\}_{\rm train}$, no optimization)
\State Predict $\btheta_{\rm pred}=\btheta_{\phi^\star}(\nu)$ in one forward pass~\eqref{eq:online_prediction}.
\State Solve $\mathbf K(\btheta_{\rm pred},\nu)\,\mathbf U_{\rm pred}=\mathbf F(\btheta_{\rm pred},\nu)$; compute $\eta(\btheta_{\rm pred};\nu)$.
\State \Return predicted mesh $\btheta_{\rm pred}$, solution $\mathbf U_{\rm pred}$, estimator $\eta$.
\end{algorithmic}
\end{algorithm}

\section{Numerical experiments}\label{sec:numerical_examples}

We compare two discretizations at a common number of degrees of freedom: throughout, \(N\) denotes the number of elements per direction, and the uniform and adapted meshes share \(p\), the continuity, and hence the number of degrees of freedom at every level. The solution \(u_h\) is computed on the uniform mesh of the given
refinement level; the \(r\)-adaptive solution \(u_\theta\) is computed on the
mesh predicted by the trained density network \(\mathcal G_{\phi^\star}\)
via~\eqref{eq:online_prediction}. For problems with a known exact solution, we
use it as the reference solution \(u^\ast\); otherwise we use a high-resolution
reference solution. The primary accuracy metric is the relative
\(H^1\)-seminorm error
\begin{equation}
\label{eq:relative_h1_error}
\Hrel(\nu)
=
\frac{|u^\ast-u_{(\cdot),\nu}|_{H^1(\Omega)}}
     {|u^\ast|_{H^1(\Omega)}},
\end{equation}
and we report the residual-based effectivity index
\begin{equation}
\label{eq:effectivity_index}
\Ieff(\nu)
=
\frac{\eta(\btheta;\nu)}
     {|u^\ast-u_{(\cdot),\nu}|_{H^1(\Omega)}}
\end{equation}
as an a posteriori diagnostic: the estimator defines the training objective, whereas~\eqref{eq:relative_h1_error} is used for offline evaluation. The reported values are medians over four random seeds at each refinement level. Convergence histories are shown in the figures, and effectivity indices are reported in the corresponding tables.

\paragraph{Training and evaluation protocol.}

For each parametric experiment, \(\mathcal P\) is sampled and split into disjoint training, validation, and test subsets (70/15/15), using the same fixed seed across degrees and methods; in Experiments~3--5 the corners of \(\mathcal P\) are additionally forced into the training set. Adam minimizes the normalized residual loss~\eqref{eq:parametric_training_loss}; validation is used for the per-level early stopping in Algorithm~\ref{alg:parametric-radapt}, and all offline metrics~\eqref{eq:relative_h1_error} are reported exclusively on the held-out test subset. All tabulated values are medians over the four seeds and the held-out test parameters of each experiment, and every inline improvement factor is the ratio of these tabulated medians. Evaluation levels outside each trained ladder (\(N\in\{128,256\}\) in Experiment~1, \(N\in\{32,48,256\}\) in Experiment~2, and \(N=64\) in Experiments~3--5) are zero-shot: the network weights are frozen after the last trained level.

\paragraph{Implementation and reproducibility.}

All experiments use splines of maximal interior continuity \(c=p-1\), except
across interface lines, where the knot multiplicity is raised to \(p\) (\(C^0\)):
the material interface \(x=1/2\) in Experiment~2 and the corner lines
\(x=1/2\), \(y=1/2\) in Experiment~4. The positional density network has two hidden layers of width \(32\) with \(\tanh\) activations (three of width \(64\) for advection--diffusion), collocated at element centers~\eqref{eq:cell_centers}, so one set of weights serves every refinement level. The mesh predictor uses the grading cap \(T\) and minimum element size \(h_{\min}\)~\eqref{eq:density_saturation}. Experiments~2--5 use \(T=5\) and \(h_{\min}=10^{-7}\); in Experiment~1 the cap increases with the refinement level, from \(T=2\) to \(T=5\) for \(p=2\) and from \(T=6\) to \(T=7\) for \(p=3\), with \(h_{\min}\in\{10^{-7},10^{-8}\}\). Sharp layers are integrated with higher-order Gauss--Legendre quadrature, used consistently for the assembly, the estimator, and the reported error norms, while the singular source is integrated analytically. The parameter sets contain \(448\) samples in Experiment~1 (\(314/67/67\)), \(100\) in Experiment~2 (\(70/15/15\)), \(2000\) in Experiment~3 (\(1400/300/300\)), and \(400\) in each of Experiments~4--5 (\(280/60/60\)); the normalization guard in~\eqref{eq:parametric_training_loss} is \(\varepsilon=10^{-12}\); Adam uses a two-stage exponentially decaying learning rate (\(10^{-2}\!\to\!10^{-4}\) in one dimension, \(10^{-3}\!\to\!10^{-4}\) in two dimensions), global-norm gradient clipping at \(1\), and a two-epoch linear warm-up at each level; early stopping monitors the validation loss with patience \(K_{\mathrm{pat}}=40\) epochs at a \(1\%\) relative-improvement tolerance, a \(40\)-epoch minimum, and a cap of \(K_{\max}=400\) epochs per level; the batch size is \(16\). In Experiment~1 the density network additionally receives the input feature \(\log(\xi_i^{c}+1/N)\), making its collocation level-aware near the singular endpoint. An optional per-instance corrector (L-BFGS-B on \(\eta^{2}\)) is implemented in the code base but disabled in all reported experiments. Reported values are medians over four random seeds, and the shaded bands in the convergence figures show the interquartile range over the pooled seeds and held-out test parameters.

\subsection{One-dimensional singular-power family}
\label{subsec:exp1_singular_power}

We consider the model problem~\eqref{eq:general_pde} with \(\sigma\equiv1\) and \(\bs\beta=\alpha=0\) on the domain \(\Omega = (0,1)\). The parameter is the singularity exponent \(\nu\), and the exact solution is the singular power
\begin{equation}
\label{eq:exp1_exact_solution}
u^\ast(x)=x^\nu,\qquad \nu\in[1.55,1.95],
\end{equation} 
so that \(-u^{\ast\prime\prime}=\nu(1-\nu)x^{\nu-2}\). Since \(u^\ast\) is only \(H^{1+\nu-1/2-\epsilon}\), a uniform mesh cannot reach the optimal rate. We use \(p\in\{2,3\}\) and \(N\) up to \(256\), and integrate the \(x^{\nu-2}\) term exactly. Figure~\ref{fig:conv-problem1-p2p3} reports the convergence: the uniform mesh is limited to the singularity rate \(N^{-(\nu-1/2)}\) (fitted slope \(\approx1.22\)), whereas the \(r\)-adaptive mesh recovers the optimal \(N^{-p}\) for both degrees, the adapted cubic error being 349 times smaller than the uniform one at \(N=64\). The effectivity indices in Table~\ref{tab:singular_summary} are not close to one, so the estimator overestimates the true \(H^1\)-seminorm error, but they remain stable under refinement and track the error trend consistently. Figures~\ref{fig:second-derivative} and~\ref{fig:solutions} show the mechanism: the adapted knots cluster at the singular endpoint \(x=0\), so the discrete second derivative follows the exact one into the high-curvature region. The training histories at \(N=4\) (Figure~\ref{fig:loss_P1}) show the residual loss decreasing to a stable plateau for both degrees, confirming convergence to a stationary point of the discrete residual.

\begin{figure}[htbp]
    \centering
    \begin{subfigure}[b]{0.49\textwidth}
        \centering
        \includegraphics[width=\linewidth]{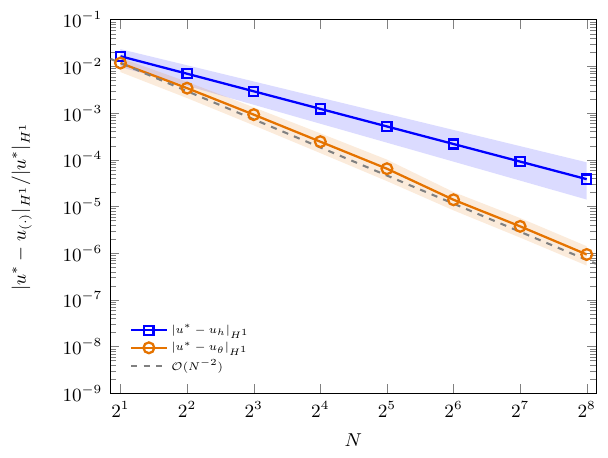}
        \caption{Quadratic B-splines (\(p=2\)).}
        \label{fig:conv-problem1-p2}
    \end{subfigure}
    \hfill
    \begin{subfigure}[b]{0.49\textwidth}
        \centering
        \includegraphics[width=\linewidth]{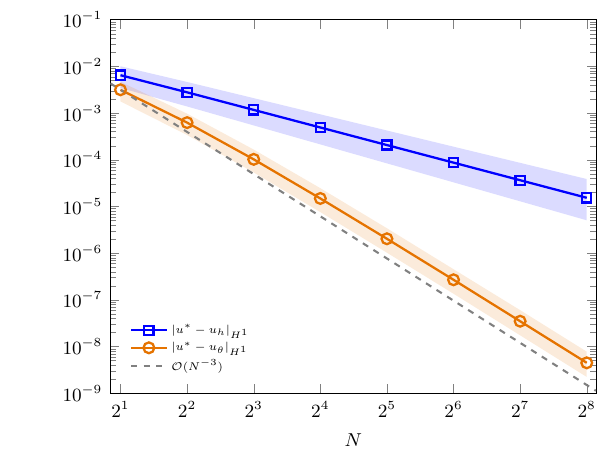}
        \caption{Cubic B-splines (\(p=3\)).}
        \label{fig:conv-problem1-p3}
    \end{subfigure}
    \caption{\textbf{Experiment~1.} Relative \(H^1\)-seminorm error versus the number of elements \(N\) for the uniform and \(r\)-adaptive meshes. Curves are medians and shaded bands the interquartile range, over seeds and the held-out test exponents.}
    \label{fig:conv-problem1-p2p3}
\end{figure}

\begin{table}[t]
\centering
\small
\caption{\textbf{Experiment 1.} Median relative $H^1$ error $|u^\ast-u_{(\cdot)}|_{H^1}/|u^\ast|_{H^1}$ and median effectivity index $I_{\mathrm{eff}}=\eta/|u^\ast-u_{(\cdot)}|_{H^1}$ per level $N$, for the uniform ($u_h$) and adapted ($u_\theta$) meshes, at $p=2$ and $p=3$. Medians pool the four seeds and the held-out test parameters.}
\label{tab:singular_summary}
\begin{adjustbox}{max width=\textwidth}
\begin{tabular}{c cccc @{\hspace{1.0em}} cccc}
\toprule
& \multicolumn{4}{c}{$p=2$} & \multicolumn{4}{c}{$p=3$} \\
\cmidrule(lr){2-5} \cmidrule(lr){6-9}
& \multicolumn{2}{c}{$H^1_{\mathrm{rel}}$} & \multicolumn{2}{c}{$I_{\mathrm{eff}}$} & \multicolumn{2}{c}{$H^1_{\mathrm{rel}}$} & \multicolumn{2}{c}{$I_{\mathrm{eff}}$} \\
\cmidrule(lr){2-3} \cmidrule(lr){4-5} \cmidrule(lr){6-7} \cmidrule(lr){8-9}
$N$ & $u_h$ & $u_\theta$ & $u_h$ & $u_\theta$ & $u_h$ & $u_\theta$ & $u_h$ & $u_\theta$ \\
\midrule
2 & $2.01{\times}10^{-2}$ & $1.41{\times}10^{-2}$ & $14.43$ & $13.49$ & $8.14{\times}10^{-3}$ & $3.97{\times}10^{-3}$ & $28.45$ & $25.32$ \\
4 & $8.69{\times}10^{-3}$ & $4.01{\times}10^{-3}$ & $14.37$ & $12.00$ & $3.57{\times}10^{-3}$ & $7.88{\times}10^{-4}$ & $27.93$ & $18.95$ \\
8 & $3.73{\times}10^{-3}$ & $1.10{\times}10^{-3}$ & $14.35$ & $11.22$ & $1.53{\times}10^{-3}$ & $1.30{\times}10^{-4}$ & $27.91$ & $14.54$ \\
16 & $1.60{\times}10^{-3}$ & $2.96{\times}10^{-4}$ & $14.35$ & $10.87$ & $6.58{\times}10^{-4}$ & $1.91{\times}10^{-5}$ & $27.91$ & $10.52$ \\
32 & $6.87{\times}10^{-4}$ & $8.01{\times}10^{-5}$ & $14.34$ & $10.94$ & $2.82{\times}10^{-4}$ & $2.62{\times}10^{-6}$ & $27.91$ & $7.54$ \\
64 & $2.95{\times}10^{-4}$ & $1.71{\times}10^{-5}$ & $14.34$ & $7.88$ & $1.21{\times}10^{-4}$ & $3.47{\times}10^{-7}$ & $27.91$ & $6.78$ \\
128 & $1.26{\times}10^{-4}$ & $4.57{\times}10^{-6}$ & $14.34$ & $7.75$ & $5.19{\times}10^{-5}$ & $4.51{\times}10^{-8}$ & $27.91$ & $7.90$ \\
256 & $5.42{\times}10^{-5}$ & $1.14{\times}10^{-6}$ & $14.34$ & $7.75$ & $2.23{\times}10^{-5}$ & $5.80{\times}10^{-9}$ & $27.91$ & $7.57$ \\
\bottomrule
\end{tabular}
\end{adjustbox}
\end{table}

\begin{figure}[htbp]
    \centering
    \begin{subfigure}[b]{0.49\textwidth}
        \centering
        \includegraphics[width=\linewidth]{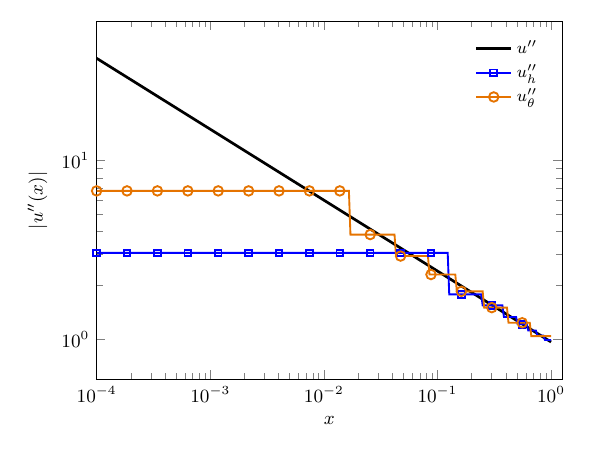}
        \caption{Quadratic B-splines (\(p=2\)).}
        \label{fig:second-derivative-p2}
    \end{subfigure}
    \hfill
    \begin{subfigure}[b]{0.49\textwidth}
        \centering
        \includegraphics[width=\linewidth]{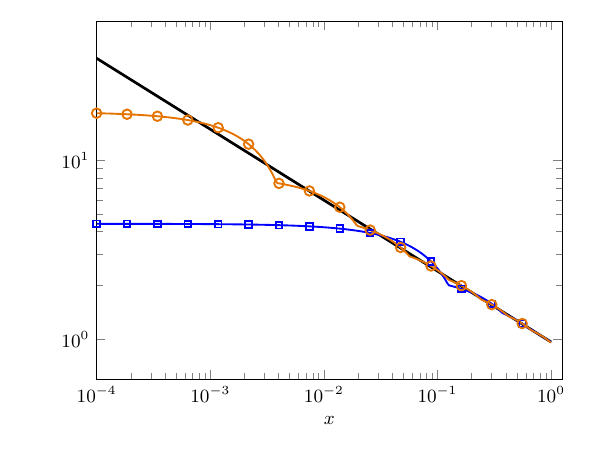}
        \caption{Cubic B-splines (\(p=3\)).}
        \label{fig:second-derivative-p3}
    \end{subfigure}
    \caption{\textbf{Experiment~1.} Log--log profiles of the exact second derivative
      and that of the numerical approximation near \(x=0\), for the uniform (blue) and
      adapted (orange) meshes, at \(\nu=1.6038\) and \(N=8\).}
    \label{fig:second-derivative}
\end{figure}

\begin{figure}[htbp]
    \centering
    \begin{subfigure}[b]{\textwidth}
        \centering
        \includegraphics[width=\linewidth]{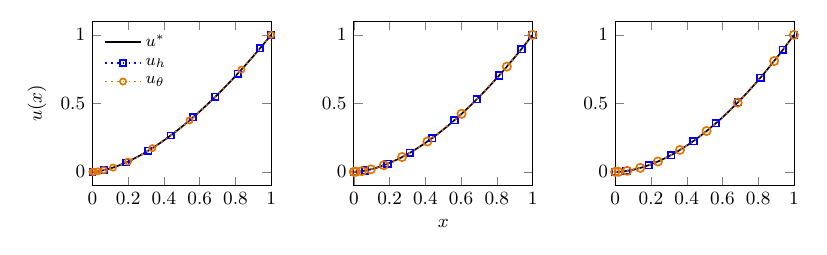}
        \caption{Quadratic B-splines (\(p=2\)).}
        \label{fig:solution_P1_p2}
    \end{subfigure}
    \vspace{0.35cm}
    \begin{subfigure}[b]{\textwidth}
        \centering
        \includegraphics[width=\linewidth]{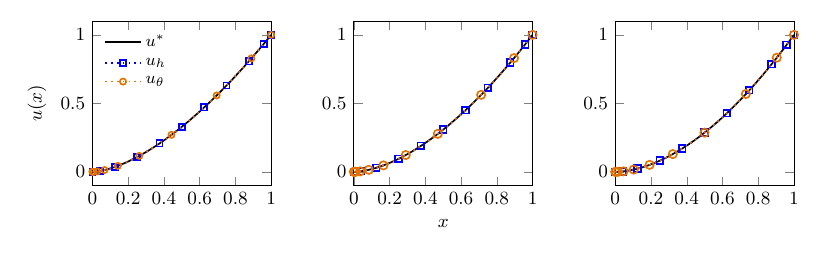}
        \caption{Cubic B-splines (\(p=3\)).}
        \label{fig:solution_P1_p3}
    \end{subfigure}
    \caption{\textbf{Experiment~1.} Exact solutions and their uniform (\(u_h\), blue)
      and \(r\)-adaptive (\(u_\theta\), orange) approximations at \(N=8\), for the
      representative exponents \(\nu\in\{1.6038,\,1.6974,\,1.8037\}\). Markers indicate
      Greville abscissae.}
    \label{fig:solutions}
\end{figure}

\begin{figure}[htbp]
    \centering
    \begin{subfigure}[t]{0.48\textwidth}
        \centering
        \includegraphics[scale=1]{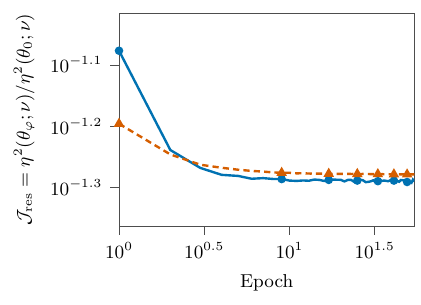}
        \caption{Quadratic B-splines (\(p=2\)).}
        \label{fig:loss_P1_p2}
    \end{subfigure}
    \hfill
    \begin{subfigure}[t]{0.48\textwidth}
        \centering
        \includegraphics[scale=1]{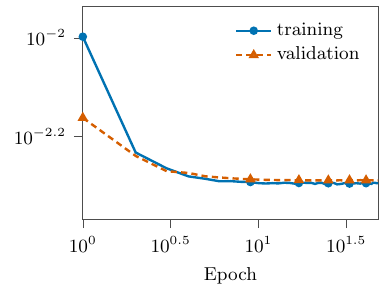}
        \caption{Cubic B-splines (\(p=3\)).}
        \label{fig:loss_P1_p3}
    \end{subfigure}
    \caption{\textbf{Experiment~1.} Training and validation losses \(\mathcal{J}_{\rm res}\) versus Adam epoch at the fixed refinement level \(N=4\), for (a)~\(p=2\) and (b)~\(p=3\).}
    \label{fig:loss_P1}
\end{figure}

\subsection{One-dimensional Helmholtz regime}
\label{subsec:exp_helmholtz}

We now consider the model problem~\eqref{eq:general_pde} on \(\Omega=(0,1)\)
with \(\bs\beta=0\), the piecewise-constant diffusion \(\sigma=\sigma_1=1\) on
\((0,1/2)\) and \(\sigma=\sigma_2=4\) on \((1/2,1)\), and the
piecewise-constant negative reaction \(\alpha=-\rho<0\) with
\(\rho=\rho_1=25\pi^{2}c^{2}\) on \((0,1/2)\) and \(\rho=\rho_2=100\pi^{2}\)
on \((1/2,1)\), which gives the indefinite Helmholtz transmission problem
\begin{equation}
  -(\sigma u')' - \rho\,u = 0 \qquad \text{in } \Omega=(0,1),
  \label{eq:helmholtz-pde}
\end{equation}
with \(u(0)=0\), the Neumann condition \(\sigma u'(1)=10\pi\), and continuity
of \(u\) and of the flux \(\sigma u'\) at the interface \(x_I=1/2\). The local
wavenumbers are \(k_j=\sqrt{\rho_j/\sigma_j}\), so that \(k_2=5\pi\) is fixed
while \(k_1=5\pi c\); the parameter \(c=k_1/k_2\in[1.5,6.0]\) is the
wavenumber contrast between the two subdomains. The exact solution is
\begin{equation}
\label{eq:exp2_exact_solution}
u^\ast_c(x)=
\begin{cases}
A\sin(k_1 x), & x\in[0,1/2],\\[2pt]
C\sin(k_2 x)+D\cos(k_2 x), & x\in[1/2,1],
\end{cases}
\end{equation}
with \((A,C,D)\) determined by the two interface conditions and the Neumann
condition; the Dirichlet condition at \(x=0\) holds by construction. The
solution is piecewise analytic, with a derivative jump at the interface
induced by the flux continuity; there is no boundary singularity, and the
uniform mesh converges at the optimal rate. Here the estimator is no longer a
certified bound, only a mesh-quality functional
(Remark~\ref{rem:coefficient_regimes}); the experiment therefore tests
whether \(r\)-adaptivity reduces the error constant, and how the method
behaves at higher wavenumber contrasts. We use \(p\in\{2,3\}\) and \(N\) up to \(256\). Three implementation details specific to this experiment depart from the generic construction. First, the mesh map: a single softmax spans \((0,1)\) and the interface knot \(x_I\) is inserted by sorting, so the element split across \(x_I\) is free rather than prescribed per segment; the two interface-adjacent sub-elements are therefore not covered by the structural floor of~\eqref{eq:density_saturation} (no trained mesh violates it; minimum observed size \(3\times10^{-6}\)), and a differentiable water-filling cap \(h_e\le 2\pi/(2.5\,k_{\mathrm{loc}})\), with \(k_{\mathrm{loc}}\) the local wavenumber, enforces a Nyquist-type sampling safeguard on coarse meshes at high contrast. Second, the reported estimator uses coefficient-independent weights (\(h_E^{2}\) and \(h_B\) in place of \(h_E^{2}/\sigma_E\) and \(h_B/\sigma_B\)); on the trained meshes this changes \(\eta\) by at most \(13\%\), and Table~\ref{tab:helmholtz_summary} inherits this convention. Third, the admissible contrasts are sampled excluding narrow bands (\(\pm0.02\)) around the resonances of the transmission problem, and training begins with a short warm phase on the highest-contrast third of the range. Both meshes share the same asymptotic rate Figure~\ref{fig:conv-problem2-p2p3}:
moving nodes cannot change the asymptotic order here, so the benefit is
confined to the error constant. Across the held-out contrasts, the adapted mesh reduces the median
\(H^1\) error at \(N=128\) by a factor of \(1.85\) for \(p=2\) and \(2.57\) for
\(p=3\). The effectivity index Table~\ref{tab:helmholtz_summary} stays
\(O(1)\), between approximately \(3.0\) and \(8.5\) across the refinement levels, drifting slowly downward as \(c\) grows (per-contrast medians, not tabulated); the estimator continues to track the true error, in the absence of a reliability guarantee. The advantage is substantially larger near the upper end of the contrast range. Near resonance, at \(c=5\) and \(N=32\), a uniform mesh fails to resolve the rapid oscillations, with a relative \(H^1\) error of \(\approx24\) for \(p=2\)---the numerical solution has entirely wrong energy---whereas the adapted mesh recovers the solution to a relative error of \(\approx0.34\) Figure~\ref{fig:solutions_P2}. In this regime, beyond the reach of the coercive theory, the predicted meshes provide the largest accuracy improvement. The training histories of Figure~\ref{fig:loss_helmholtz} reflect the warm phase noted above: over the first sixty epochs, the training loss is averaged on the highest-contrast third of the range and decreases steadily; at epoch 61, when the full range is released, it steps up---by factors of 2.1 \((p=2)\) and 2.7 \((p=3)\) in the runs shown---because the average then also covers the low-contrast tuples, on which a uniform mesh is already close to optimal and the normalized residual~\eqref{eq:parametric_training_loss} is correspondingly larger. The jump is the result of a change of averaging set rather than a loss of accuracy: the validation loss, evaluated on the full held-out set throughout, does not rise, and both curves flatten before early stopping.

\begin{figure}[htbp]
    \centering
    \begin{subfigure}[b]{0.49\textwidth}
        \centering
        \includegraphics[width=\linewidth]{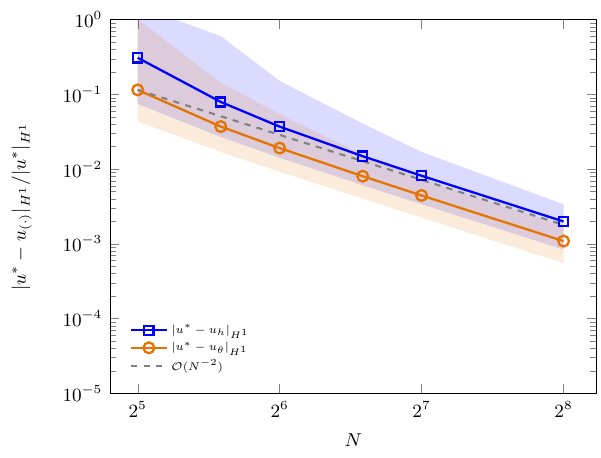}
        \caption{Quadratic B-splines (\(p=2\)).}
        \label{fig:conv-problem2-p2}
    \end{subfigure}
    \hfill
    \begin{subfigure}[b]{0.49\textwidth}
        \centering
        \includegraphics[width=\linewidth]{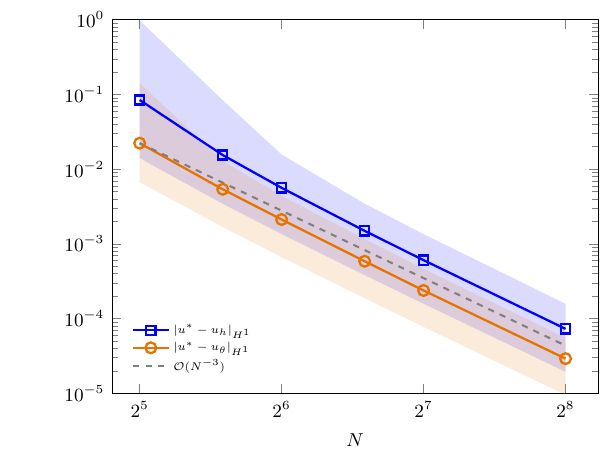}
        \caption{Cubic B-splines (\(p=3\)).}
        \label{fig:conv-problem2-p3}
    \end{subfigure}
    \caption{\textbf{Experiment~2.} Relative \(H^1\)-seminorm error versus \(N\) for the uniform and \(r\)-adaptive meshes. Curves are medians and shaded bands the interquartile range, over the four seeds and the held-out test contrasts.}
    \label{fig:conv-problem2-p2p3}
\end{figure}

\begin{figure}[htbp]
\centering
\begin{subfigure}[b]{0.48\textwidth}
    \centering
    \includegraphics[width=\linewidth]{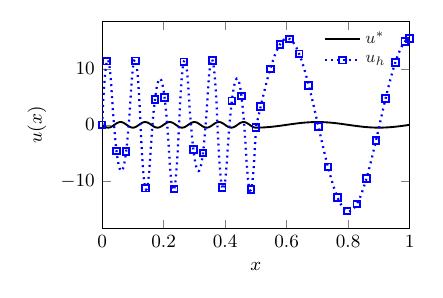}
    \caption{Uniform, \(p=2\).}
    \label{fig:P2_solution_uniform_p2}
\end{subfigure}
\hfill
\begin{subfigure}[b]{0.48\textwidth}
    \centering
    \includegraphics[width=\linewidth]{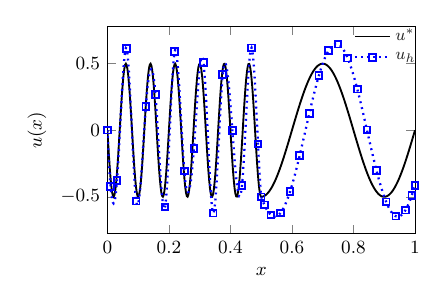}
    \caption{Uniform, \(p=3\).}
    \label{fig:P2_solution_uniform_p3}
\end{subfigure}
\vspace{0.35cm}
\begin{subfigure}[b]{0.48\textwidth}
    \centering
    \includegraphics[width=\linewidth]{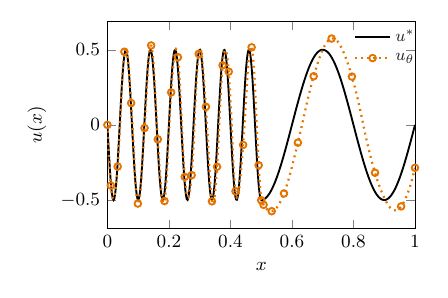}
    \caption{Adapted, \(p=2\).}
    \label{fig:P2_solution_NN_p2}
\end{subfigure}
\hfill
\begin{subfigure}[b]{0.48\textwidth}
    \centering
    \includegraphics[width=\linewidth]{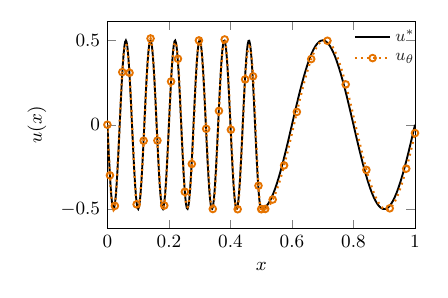}
    \caption{Adapted, \(p=3\).}
    \label{fig:P2_solution_NN_p3}
\end{subfigure}
\caption{\textbf{Experiment~2.} Exact solution and its uniform (\(u_h\), blue) and \(r\)-adaptive (\(u_\theta\), orange) approximations near resonance, at \(c=5\) and \(N=32\).}
\label{fig:solutions_P2}
\end{figure}

\begin{table}[t]
\centering
\small
\caption{\textbf{Experiment 2.} Median relative $H^1$ error $|u^\ast-u_{(\cdot)}|_{H^1}/|u^\ast|_{H^1}$ and median effectivity index $I_{\mathrm{eff}}=\eta/|u^\ast-u_{(\cdot)}|_{H^1}$ per level $N$, for the uniform ($u_h$) and adapted ($u_\theta$) meshes, at $p=2$ and $p=3$. Medians pool the four seeds and the held-out test parameters.}
\label{tab:helmholtz_summary}
\begin{adjustbox}{max width=\textwidth}
\begin{tabular}{c cccc @{\hspace{1.0em}} cccc}
\toprule
& \multicolumn{4}{c}{$p=2$} & \multicolumn{4}{c}{$p=3$} \\
\cmidrule(lr){2-5} \cmidrule(lr){6-9}
& \multicolumn{2}{c}{$H^1_{\mathrm{rel}}$} & \multicolumn{2}{c}{$I_{\mathrm{eff}}$} & \multicolumn{2}{c}{$H^1_{\mathrm{rel}}$} & \multicolumn{2}{c}{$I_{\mathrm{eff}}$} \\
\cmidrule(lr){2-3} \cmidrule(lr){4-5} \cmidrule(lr){6-7} \cmidrule(lr){8-9}
$N$ & $u_h$ & $u_\theta$ & $u_h$ & $u_\theta$ & $u_h$ & $u_\theta$ & $u_h$ & $u_\theta$ \\
\midrule
32 & $3.09{\times}10^{-1}$ & $1.13{\times}10^{-1}$ & $2.99$ & $5.30$ & $8.42{\times}10^{-2}$ & $2.31{\times}10^{-2}$ & $3.47$ & $5.22$ \\
48 & $7.93{\times}10^{-2}$ & $3.67{\times}10^{-2}$ & $5.42$ & $7.37$ & $1.54{\times}10^{-2}$ & $5.49{\times}10^{-3}$ & $5.23$ & $6.40$ \\
64 & $3.71{\times}10^{-2}$ & $1.89{\times}10^{-2}$ & $6.62$ & $7.98$ & $5.62{\times}10^{-3}$ & $2.09{\times}10^{-3}$ & $5.77$ & $6.51$ \\
96 & $1.49{\times}10^{-2}$ & $7.97{\times}10^{-3}$ & $7.35$ & $8.28$ & $1.49{\times}10^{-3}$ & $5.76{\times}10^{-4}$ & $6.17$ & $6.81$ \\
128 & $8.15{\times}10^{-3}$ & $4.40{\times}10^{-3}$ & $7.55$ & $8.40$ & $6.05{\times}10^{-4}$ & $2.35{\times}10^{-4}$ & $6.31$ & $6.94$ \\
256 & $1.99{\times}10^{-3}$ & $1.08{\times}10^{-3}$ & $7.72$ & $8.51$ & $7.28{\times}10^{-5}$ & $2.86{\times}10^{-5}$ & $6.45$ & $7.05$ \\
\bottomrule
\end{tabular}
\end{adjustbox}
\end{table}

\begin{figure}[htbp]
    \centering
    \begin{subfigure}[t]{0.48\textwidth}
        \centering
        \includegraphics[scale=1]{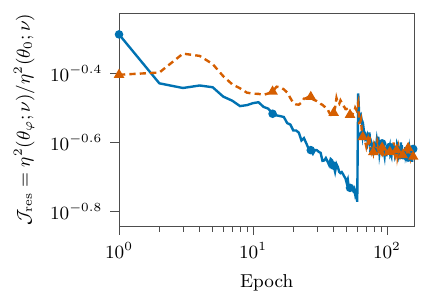}
        \caption{Quadratic B-splines (\(p=2\)).}
        \label{fig:loss_helmholtz_p2}
    \end{subfigure}
    \hfill
    \begin{subfigure}[t]{0.48\textwidth}
        \centering
        \includegraphics[scale=1]{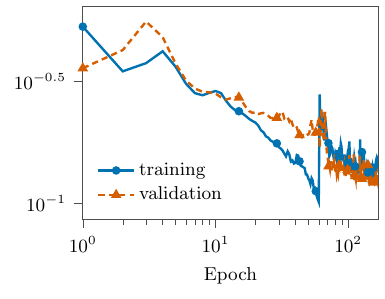}
        \caption{Cubic B-splines (\(p=3\)).}
        \label{fig:loss_helmholtz_p3}
    \end{subfigure}
    \caption{\textbf{Experiment~2.} Training and validation losses \(\mathcal{J}_{\rm res}\) versus Adam epoch at the fixed refinement level \(N=64\), for (a)~\(p=2\) and (b)~\(p=3\).}
    \label{fig:loss_helmholtz}
\end{figure}

\subsection{Two-dimensional arctangent family}
\label{subsec:exp3_arctangent}

We now consider the model problem~\eqref{eq:general_pde} with \(\sigma\equiv1\) and \(\bs\beta=\alpha=0\) on \(\Omega=(0,1)^2\), the first two-dimensional test. The manufactured solution is the smooth arctangent product
\begin{equation}
\label{eq:exp3_exact_solution}
u_\nu^\ast(x,y)=u_1(x)\,u_2(y),
\qquad \nu=(t,s_1,s_2),
\qquad
u_j(w)=\arctan\!\bigl(t(w-s_j)\bigr)+\arctan(ts_j),
\end{equation}
with \(-\Delta u_\nu^\ast=f_\nu\) in \(\Omega\), homogeneous Dirichlet conditions on \(\{x=0\}\cup\{y=0\}\), where \(u_\nu^\ast\) vanishes identically, and non-homogeneous Neumann data \(g_\nu=\sigma\nabla u_\nu^\ast\cdot\bn\) on \(\{x=1\}\cup\{y=1\}\); the Neumann faces contribute the corresponding boundary residual terms of~\eqref{eq:robust_estimator}. Here \(t\) sets the steepness of an internal layer and \((s_1,s_2)\) its position. The solution is smooth and the layer is axis-aligned, which suits tensor-product knot motion, so the test isolates the constant-factor benefit in two dimensions. We use \(p\in\{2,3\}\) and \(N\) up to \(64\) per direction, on the parameter grid of~\cite{aballay2025radaptivefiniteelementmethod}. Both meshes converge at the optimal rate \(\mathcal O(N^{-p})\) Figure~\ref{fig:conv-problem3-p2p3}, as expected for a smooth solution, and the adapted mesh again reduces the error constant: at \(N=32\) the improvement factor is \(6.1\) for \(p=2\) and \(31.6\) for \(p=3\). The effectivity indices in Table~\ref{tab:arctan_summary} are not close to one, so the estimator overestimates the true \(H^1\)-seminorm error, but they remain stable under refinement and track the error trend consistently. Since the adapted meshes are plausibly anisotropic, the reliability of Proposition~\ref{prop:dual_reliability} is
only heuristic here (Remark~\ref{rem:anisotropy}), consistent with the fixed-DOF
reading. Figure~\ref{fig:P3_solutions} shows the mechanism: the predicted knot lines concentrate along the internal layer, resolving it in both directions. The training histories at \(N=4\) Figure~\ref{fig:loss_arctangent} show the residual loss decreasing to a stable plateau for both degrees.

\begin{figure}[htbp]
    \centering
    \begin{subfigure}[b]{0.49\textwidth}
        \centering
        \includegraphics[width=\linewidth]{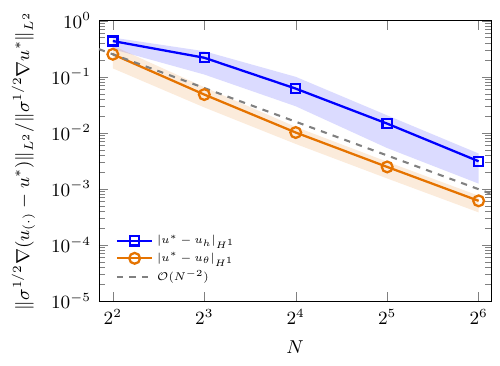}
        \caption{Quadratic B-splines (\(p=2\)).}
        \label{fig:conv-problem3-p2}
    \end{subfigure}
    \hfill
    \begin{subfigure}[b]{0.49\textwidth}
        \centering
        \includegraphics[width=\linewidth]{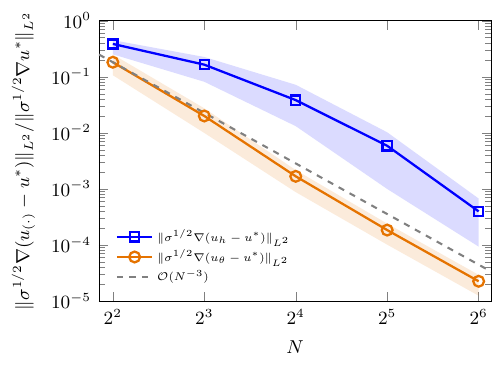}
        \caption{Cubic B-splines (\(p=3\)).}
        \label{fig:conv-problem3-p3}
    \end{subfigure}
    \caption{\textbf{Experiment~3.} Relative weighted \(H^1\)-seminorm error versus \(N\) for the uniform and \(r\)-adaptive meshes. Curves are medians and shaded bands the interquartile range, over seeds and the held-out test parameters.}
    \label{fig:conv-problem3-p2p3}
\end{figure}

\begin{table}[t]
\centering
\small
\caption{\textbf{Experiment 3.} Median relative $H^1$ error $|u^\ast-u_{(\cdot)}|_{H^1}/|u^\ast|_{H^1}$ and median effectivity index $I_{\mathrm{eff}}=\eta/|u^\ast-u_{(\cdot)}|_{H^1}$ per level $N$, for the uniform ($u_h$) and adapted ($u_\theta$) meshes, at $p=2$ and $p=3$. Medians pool the four seeds and the held-out test parameters.}
\label{tab:arctan_summary}
\begin{adjustbox}{max width=\textwidth}
\begin{tabular}{c cccc @{\hspace{1.0em}} cccc}
\toprule
& \multicolumn{4}{c}{$p=2$} & \multicolumn{4}{c}{$p=3$} \\
\cmidrule(lr){2-5} \cmidrule(lr){6-9}
& \multicolumn{2}{c}{$H^1_{\mathrm{rel}}$} & \multicolumn{2}{c}{$I_{\mathrm{eff}}$} & \multicolumn{2}{c}{$H^1_{\mathrm{rel}}$} & \multicolumn{2}{c}{$I_{\mathrm{eff}}$} \\
\cmidrule(lr){2-3} \cmidrule(lr){4-5} \cmidrule(lr){6-7} \cmidrule(lr){8-9}
$N$ & $u_h$ & $u_\theta$ & $u_h$ & $u_\theta$ & $u_h$ & $u_\theta$ & $u_h$ & $u_\theta$ \\
\midrule
4 & $4.30{\times}10^{-1}$ & $2.45{\times}10^{-1}$ & $8.67$ & $10.35$ & $3.80{\times}10^{-1}$ & $1.69{\times}10^{-1}$ & $9.30$ & $11.44$ \\
8 & $2.23{\times}10^{-1}$ & $4.72{\times}10^{-2}$ & $6.63$ & $11.95$ & $1.65{\times}10^{-1}$ & $1.90{\times}10^{-2}$ & $6.75$ & $10.98$ \\
16 & $6.09{\times}10^{-2}$ & $9.84{\times}10^{-3}$ & $7.28$ & $13.62$ & $3.56{\times}10^{-2}$ & $1.68{\times}10^{-3}$ & $5.53$ & $12.65$ \\
32 & $1.44{\times}10^{-2}$ & $2.36{\times}10^{-3}$ & $9.15$ & $13.93$ & $5.88{\times}10^{-3}$ & $1.86{\times}10^{-4}$ & $6.05$ & $12.84$ \\
64 & $3.09{\times}10^{-3}$ & $5.86{\times}10^{-4}$ & $10.46$ & $14.00$ & $3.93{\times}10^{-4}$ & $2.26{\times}10^{-5}$ & $8.09$ & $12.83$ \\
\bottomrule
\end{tabular}
\end{adjustbox}
\end{table}

\begin{figure}[htbp]
    \centering
    \includegraphics[width=\linewidth]{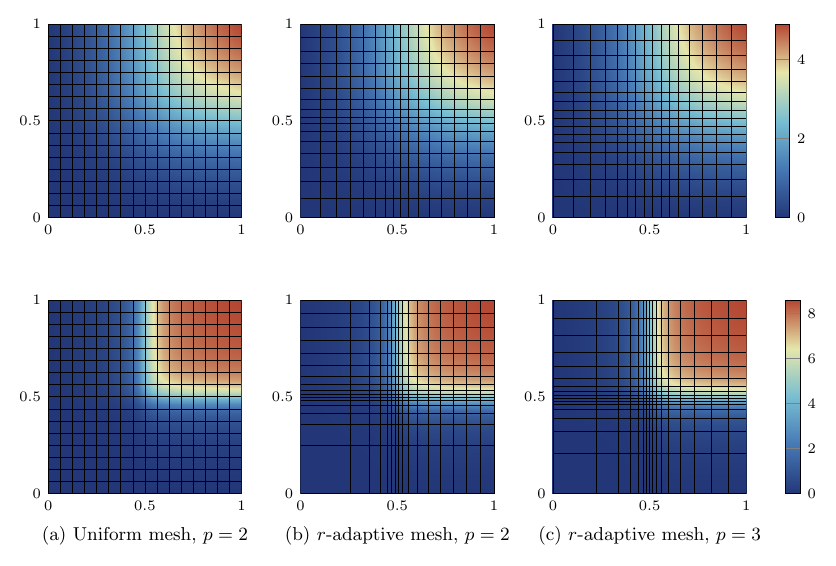}
    \caption{\textbf{Experiment~3.} Predicted meshes and solutions at \(N=16\) for two
      representative values of \(\nu=(t,s_1,s_2)\), one per row: upper panels~(a)--(c)
      correspond to \(\nu=(3.92,\,0.54,\,0.54)\) and lower panels~(a)--(c) to
      \((19.12,\,0.46,\,0.46)\).}
    \label{fig:P3_solutions}
\end{figure}

\begin{figure}[htbp]
    \centering
    \begin{subfigure}[t]{0.48\textwidth}
        \centering
        \includegraphics[scale=1]{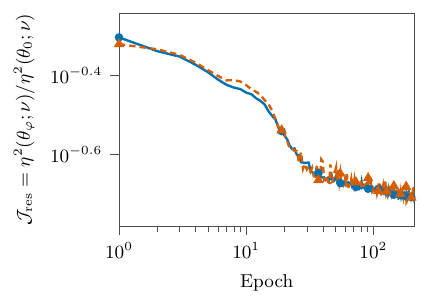}
        \caption{Quadratic B-splines (\(p=2\)).}
        \label{fig:loss_arctangent_p2}
    \end{subfigure}
    \hfill
    \begin{subfigure}[t]{0.48\textwidth}
        \centering
        \includegraphics[scale=1]{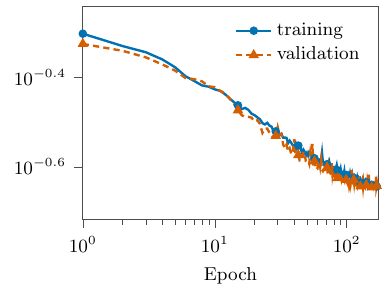}
        \caption{Cubic B-splines (\(p=3\)).}
        \label{fig:loss_arctangent_p3}
    \end{subfigure}
    \caption{\textbf{Experiment~3.} Training and validation losses \(\mathcal{J}_{\rm res}\) versus Adam epoch at the fixed refinement level \(N=4\), for (a)~\(p=2\) and (b)~\(p=3\).}
    \label{fig:loss_arctangent}
\end{figure}

\subsection{Two-dimensional L-shaped domain}
\label{subsec:exp4_lshape}
We now consider the model problem~\eqref{eq:general_pde} with \(\bs\beta=\alpha=0\) on the L-shaped domain \(\Omega=(0,1)^2\setminus\bigl([1/2,1]\times[0,1/2]\bigr)\),
\begin{equation}
\label{eq:lshape_pde}
-\nabla\cdot(\sigma\nabla u)=1 \quad\text{in }\Omega,
\qquad u=0 \quad\text{on }\partial\Omega.
\end{equation}
The diffusion coefficient is piecewise constant, taking the values \(1\), \(\sigma_1\), \(\sigma_2\) on the three sub-regions around the re-entrant corner \((1/2,1/2)\), with \(\sigma_1,\sigma_2\in[10^{-1},10^{1}]\). The corner forces a \(u\sim r^{2/3}\) singularity~\cite{grisvard1985elliptic} that caps the convergence rate on any quasi-uniform mesh, whatever the degree. The test is how much accuracy node relocation alone can recover at a fixed number of degrees of freedom. Instead of a conforming multipatch mesh, we use an immersed, fictitious-domain IGA on the embedding square \((0,1)^2\), in the spirit of the Finite Cell Method~\cite{parvizian2007finite,schillinger2015finitecell,schillinger2013isogeometric}: a single tensor-product spline space in which the corner lines \(x=1/2\), \(y=1/2\) are made \(C^{0}\) by raising the knot \(1/2\) to multiplicity \(p\), while the removed quadrant and outer boundary are pinned to zero. The cut lies on fixed knot lines, so the masking is exact and the mesh parametrization is that of the other experiments, with four knot-spacing vectors (one on each side of the two corner lines). Because the \(C^{0}\) cut makes the discretization non-conforming and the energy seminorm non-monotone~\cite{buffa2022mathematical}, we measure the error directly (a weak cut imposition~\cite{embar2010imposing} with the CutFEM error analysis~\cite{burman2012fictitious} would be sharper). Lacking a closed form, we take as reference \(u^\ast\) a heavily graded degree-\(p=5\) immersed solution carrying the same \(C^{0}\) cut, with an estimated accuracy of \(\lesssim2\times10^{-5}\) relative \(H^1\) within the immersed family (supported by a self-convergence study of the reference discretization included with the code); this lies at least one order of magnitude below the smallest reported error (\(2.34\times10^{-4}\); \(p=3\), \(N=64\), adapted) and more than two orders below the uniform-mesh errors, so the reference perturbs the finest adapted values by at most \(\sim9\%\) and does not affect the comparisons. We measure the \(\sigma\)-weighted \(H^1\) seminorm,
\begin{equation}
\label{eq:lshape_direct_metric}
|v|_{H^1}:=\Bigl(\int_\Omega\sigma\,|\nabla v|^{2}\,\dx\Bigr)^{1/2},
\qquad
H^1_\mathrm{rel}:=\frac{|u_{(\cdot)}-u^\ast|_{H^1}}{|u^\ast|_{H^1}},
\end{equation}
computed by Gauss--Legendre quadrature masked to the L-shape; for \(\bs\beta=\alpha=0\) it coincides with the energy norm~\eqref{eq:energy_norm}.
We use \(p\in\{2,3\}\), continuation levels \(N\in\{4,8,16,32\}\) per axis with evaluation to \(N=64\), and \(T=5\), \(h_{\min}=10^{-7}\). The uniform mesh converges at the singularity-limited rate dictated by the corner---fitted slopes over \(N\in[16,64]\) of \(\approx0.73\) for both degrees (Figure~\ref{fig:conv-problem4-p2p3})---while the adapted meshes reach substantially lower errors, with effective orders (endpoint slopes over \(N\in[4,64]\)) of \(1.87\) (\(p=2\)) and \(2.06\) (\(p=3\)) and error reductions at \(N=64\) of \(16.1\) and \(27.5\) times (Table~\ref{tab:lshape_summary}). The gain is not uniform in the budget: at the coarsest level \(N=4\) strong grading starves the bulk of the domain and the adapted error slightly exceeds the uniform one, the crossover to a net gain appearing from \(N=8\). For \(p=3\) the local slope rises to \(\approx3.05\) over the trained levels and then drops to \(1.14\) in the final \(N=32\to64\) step, where evaluation extends beyond the training range, though the cubic error stays below the quadratic at every level. The reduced order is not an artifact of relocation itself---Experiment~1 recovers the full \(N^{-p}\) rate for a point singularity---but structural: a single knot line at \(x=1/2\) cannot localize the corner without refining an entire strip, so a finite budget acts as an algebraically graded mesh of bounded effective order. Since proper grading recovers optimal rates for corner singularities~\cite{babuska1979direct,apel1999anisotropic,schwab1998phpfem}, the cap reflects the grading a tensor-product relocation affords rather than a fundamental limit, and a conforming multipatch discretization would inherit it through interface conformity. Three causes remain entangled---the tensor-product knot-line structure, the approximate \(C^{0}\) enforcement, and the grading cap (\(T\), \(h_{\min}\); Section~\ref{sec:mesh_param})---and we leave their separation, by a conforming discretization or a sensitivity study in \(T\) and \(h_{\min}\), to future work. We therefore claim only what the data show: the corner caps the adapted effective order near \(2\) for \(p=2\) and reduces the local \(p=3\) order toward \(1\) at the finest levels. In this immersed setting the estimator overestimates the energy error: the effectivity indices range from about \(6\) to \(48\) (Table~\ref{tab:lshape_summary}), because the \(C^{0}\) cut introduces flux-jump terms that inflate \(\eta\). The inflation is strongest on the coarse, strongly graded \(p=3\) meshes---where \(I_{\mathrm{eff}}\) reaches \(48\)---and relaxes under refinement, as the jump contribution scales differently from the true error and its relative weight shrinks. We therefore read \(\eta\) here as a differentiable mesh-quality functional rather than a sharp bound, and assess accuracy through the direct error~\eqref{eq:lshape_direct_metric}. Figures~\ref{fig:P4_solutions} and~\ref{fig:P4_gradients} show the mechanism---the four spacing vectors concentrate knots at the re-entrant corner, where the gradient is steepest---and the training histories at \(N=8\) (Figure~\ref{fig:loss_lshape}) show the residual loss decreasing to a stable plateau for both degrees.
\begin{figure}[htbp]
    \centering
    \begin{subfigure}[b]{0.49\textwidth}
        \centering
        \includegraphics[width=\linewidth]{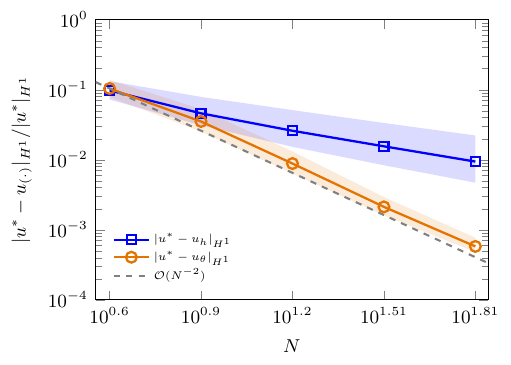}
        \caption{Quadratic B-splines (\(p=2\)).}
        \label{fig:conv-problem4-p2}
    \end{subfigure}
    \hfill
    \begin{subfigure}[b]{0.49\textwidth}
        \centering
        \includegraphics[width=\linewidth]{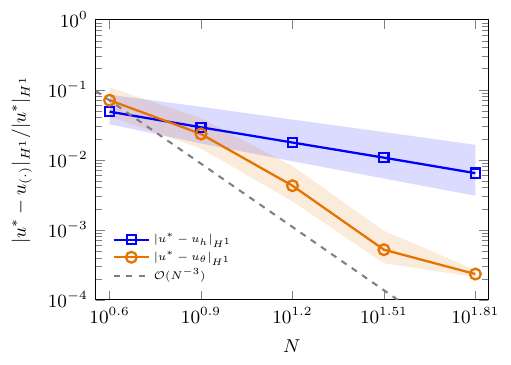}
        \caption{Cubic B-splines (\(p=3\)).}
        \label{fig:conv-problem4-p3}
    \end{subfigure}
    \caption{\textbf{Experiment~4.} Relative error in the \(\sigma\)-weighted
      \(H^1\) seminorm~\eqref{eq:lshape_direct_metric} versus the number of
      elements per axis \(N\) for the uniform and \(r\)-adaptive meshes (median
      over seeds and the held-out test parameters, measured against the
      high-degree immersed reference), trained with batch size~\(16\). The
      uniform rate is singularity-limited (fitted slopes \(\approx0.73\) over \(N\in[16,64]\))}
    \label{fig:conv-problem4-p2p3}
\end{figure}

\begin{figure}[htbp]
    \centering
    \includegraphics[width=\linewidth]{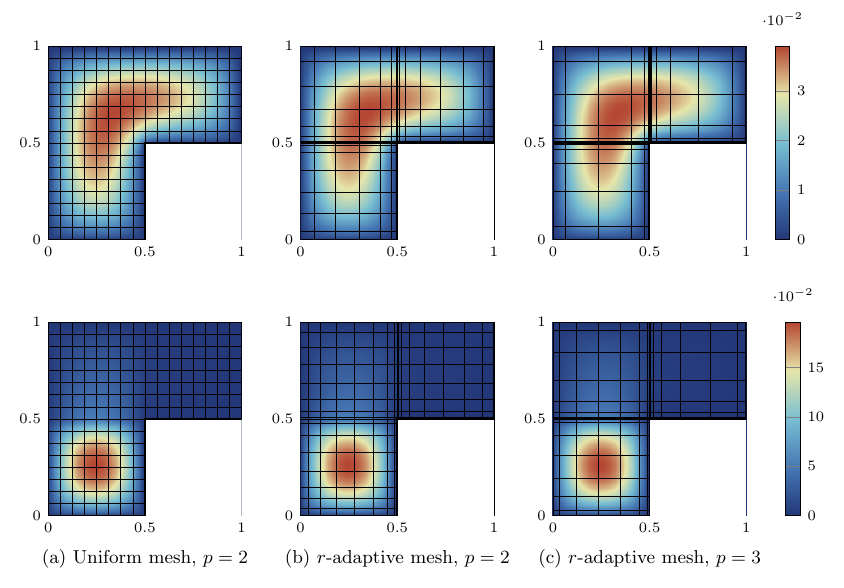}
    \caption{\textbf{Experiment~4.} Predicted meshes and solutions at \(N=16\) for two representative parameter pairs, one per row: upper panels~(a)--(c) correspond to \((\sigma_1,\sigma_2)=(0.89,\,0.89)\) and lower panels~(a)--(c) to \((10,\,0.1)\).}
      \label{fig:P4_solutions}
\end{figure}

\begin{figure}[htbp]
    \centering
    \includegraphics[width=\linewidth]{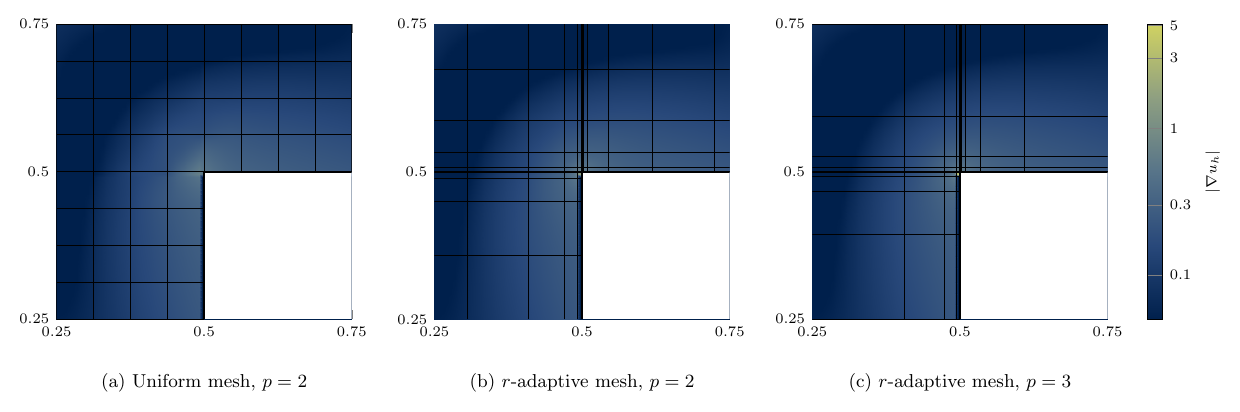}
    \caption{\textbf{Experiment~4.} Detail of the solution gradient near the re-entrant corner at \((\sigma_1,\sigma_2)=(0.89,\,0.89)\) and \(N=16\), with the knot lines superimposed. Panels~(a)--(c) share a common color scale.}
    \label{fig:P4_gradients}
\end{figure}

\begin{figure}[htbp]
    \centering
    \begin{subfigure}[t]{0.48\textwidth}
        \centering
        \includegraphics[scale=1]{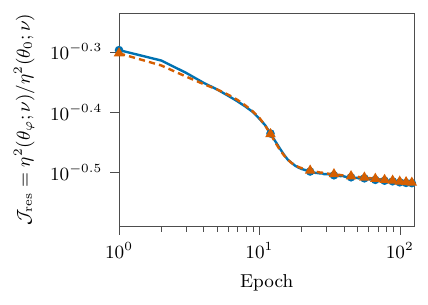}
        \caption{Quadratic B-splines (\(p=2\)).}
        \label{fig:loss_lshape_p2}
    \end{subfigure}
    \hfill
    \begin{subfigure}[t]{0.48\textwidth}
        \centering
        \includegraphics[scale=1]{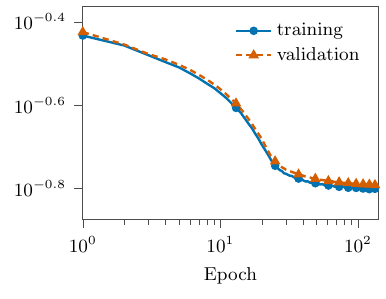}
        \caption{Cubic B-splines (\(p=3\)).}
        \label{fig:loss_lshape_p3}
    \end{subfigure}
    \caption{\textbf{Experiment~4.} Training and validation losses \(\mathcal{J}_{\rm res}\) versus Adam epoch at the fixed refinement level \(N=8\), for (a)~\(p=2\) and (b)~\(p=3\).}
    \label{fig:loss_lshape}
\end{figure}

\begin{table}[t]
\centering
\small
\caption{\textbf{Experiment 4.} Median relative $H^1$ error $|u^\ast-u_{(\cdot)}|_{H^1}/|u^\ast|_{H^1}$ and median effectivity index $I_{\mathrm{eff}}=\eta/|u^\ast-u_{(\cdot)}|_{H^1}$ per level $N$, for the uniform ($u_h$) and adapted ($u_\theta$) meshes, at $p=2$ and $p=3$.}
\label{tab:lshape_summary}
\begin{adjustbox}{max width=\textwidth}
\begin{tabular}{c cccc @{\hspace{1.0em}} cccc}
\toprule
& \multicolumn{4}{c}{$p=2$} & \multicolumn{4}{c}{$p=3$} \\
\cmidrule(lr){2-5} \cmidrule(lr){6-9}
& \multicolumn{2}{c}{$H^1_{\mathrm{rel}}$} & \multicolumn{2}{c}{$I_{\mathrm{eff}}$} & \multicolumn{2}{c}{$H^1_{\mathrm{rel}}$} & \multicolumn{2}{c}{$I_{\mathrm{eff}}$} \\
\cmidrule(lr){2-3} \cmidrule(lr){4-5} \cmidrule(lr){6-7} \cmidrule(lr){8-9}
$N$ & $u_h$ & $u_\theta$ & $u_h$ & $u_\theta$ & $u_h$ & $u_\theta$ & $u_h$ & $u_\theta$ \\
\midrule
4  & $9.78{\times}10^{-2}$ & $1.04{\times}10^{-1}$ & $8.8$ & $11.9$ & $4.87{\times}10^{-2}$ & $7.09{\times}10^{-2}$ & $11.1$ & $41.9$ \\
8  & $4.61{\times}10^{-2}$ & $3.53{\times}10^{-2}$ & $7.1$ & $14.5$ & $2.91{\times}10^{-2}$ & $2.34{\times}10^{-2}$ & $9.3$  & $48.0$ \\
16 & $2.60{\times}10^{-2}$ & $8.84{\times}10^{-3}$ & $6.1$ & $10.9$ & $1.76{\times}10^{-2}$ & $4.27{\times}10^{-3}$ & $9.0$  & $19.9$ \\
32 & $1.56{\times}10^{-2}$ & $2.13{\times}10^{-3}$ & $5.9$ & $10.5$ & $1.07{\times}10^{-2}$ & $5.21{\times}10^{-4}$ & $9.0$  & $8.8$  \\
64 & $9.43{\times}10^{-3}$ & $5.80{\times}10^{-4}$ & $5.9$ & $12.4$ & $6.44{\times}10^{-3}$ & $2.34{\times}10^{-4}$ & $9.2$  & $5.8$  \\
\bottomrule
\end{tabular}
\end{adjustbox}
\end{table}

\subsection{Two-dimensional advection--diffusion boundary layer}
\label{subsec:exp5_advection_diffusion}

We finally consider the model problem~\eqref{eq:general_pde} with constant diffusion
\(\sigma\equiv\varepsilon\), advection \(\bs\beta=(b,0)^\top\), and \(\alpha=0\) on
\(\Omega=(0,1)^2\),
\begin{equation}
\label{eq:exp5_strong}
-\varepsilon\,\Delta u_{\bs\nu}
+\bs\beta\cdot\nabla u_{\bs\nu}=f_{\bs\nu}
\quad\text{in }\Omega,
\qquad u_{\bs\nu}=0 \ \text{on }\partial\Omega,
\end{equation}
parametrized by \(\bs\nu=(\ell_\varepsilon,b)\in[-2.0,-1.5]\times[0.5,2.0]\) with
\(\varepsilon=10^{\ell_\varepsilon}\). The layer width
\(\delta_{\bs\nu}=\varepsilon/b\) varies by an order of magnitude across the
parameter set. Since the diffusion is small, the energy constants degrade with
\(\varepsilon\); we therefore read this as a fixed-resolution adaptation test, not a
claim of \(\varepsilon\)-robustness (Remark~\ref{rem:coefficient_regimes})
\cite{verfurth2005convectiondiffusion,sangalli2008robust,apel2011anisotropic}. The manufactured solution
\begin{equation}
\label{eq:exp5_exact}
u^*_{\bs\nu}(x,y)=x\bigl(1-e^{(x-1)/\delta_{\bs\nu}}\bigr)\sin(\pi y)
\end{equation}
exhibits an outflow boundary layer near \(x=1\), with characteristic thickness
\(\delta_{\bs\nu}\). Substituting it into the operator gives the source
\begin{equation}
\label{eq:exp5_source}
f_{\bs\nu}(x,y)=\sin(\pi y)\,\Bigl[\,b\bigl(1+e^{z}\bigr)
+\varepsilon\,\pi^2\,x\bigl(1-e^{z}\bigr)\,\Bigr],
\qquad z=\frac{x-1}{\delta_{\bs\nu}},
\end{equation}
where the steep \((xb^2/\varepsilon)e^{z}\) terms cancel because
\(\delta_{\bs\nu}=\varepsilon/b\). The solution vanishes on all four edges, so the
Dirichlet condition needs no lifting, and we integrate the sharp layer with a
high-order Gauss--Legendre rule. We use \(p\in\{2,3\}\) and \(N\) up to \(64\) per direction, and report the
effectivity index only on the diffusion-dominated sub-range
\(\ell_\varepsilon\in[-1.75,-1.5]\), where the dependence of the estimator
constants on \(\varepsilon\) is mildest. The results are consistent with the expected behavior
(Figure~\ref{fig:conv-problem5-p2p3}). For the convection-dominated cases the coarse
uniform mesh has a large cell P\'eclet number and oscillates near the layer, while
the adapted mesh resolves it. As \(N\) increases, the uniform-mesh error decreases,
but the adapted mesh continues to provide a smaller \(H^1\)-seminorm error. The
finest level \(N=64\) lies beyond the continuation training range (\(N\le32\)), so
it probes the zero-shot extrapolation of the positional network. The network covers
the full range of layer widths in a single forward pass. The effectivity indices in Table~\ref{tab:advdiff_summary} are \(O(1)\)--\(O(10)\),
stable under refinement, and comparable to the other experiments. Two scales
interact here: the estimator controls the energy norm, which for this problem
carries the weight \(\sigma=\varepsilon\), i.e.\ \(\|e\|_{\mathcal E}=\varepsilon^{1/2}|e|_{H^1}\),
while the index in~\eqref{eq:effectivity_index} divides by the unweighted \(H^1\)
seminorm. A sharp energy-norm estimator would thus give
\(\Ieff\sim\varepsilon^{1/2}\in[0.13,0.18]\) on the reported sub-range; the observed
values exceed this by one to two orders, quantifying the \(\varepsilon\)-dependence
of the estimator constants anticipated in
Remark~\ref{rem:coefficient_regimes}: measured against the energy-norm error, the
overestimation factor \(\varepsilon^{-1/2}\,\Ieff\) is \(\sim10\)--\(70\), comparable
to the corner-singular Experiment~4. As in Experiment~3, the layer-resolving meshes here are
anisotropic, so Proposition~\ref{prop:dual_reliability} is again heuristic
(Remark~\ref{rem:anisotropy}). Figures~\ref{fig:P5_solutions} and~\ref{fig:P5_cross_section} show the mechanism:
the predicted knots concentrate at the outflow layer, where the uniform mesh
oscillates. The training histories at \(N=4\) (Figure~\ref{fig:loss_advdiff}) show the residual loss decreasing to a stable plateau for both degrees.

\begin{figure}[htbp]
    \centering
    \begin{subfigure}[b]{0.49\textwidth}
        \centering
        \includegraphics[width=\linewidth]{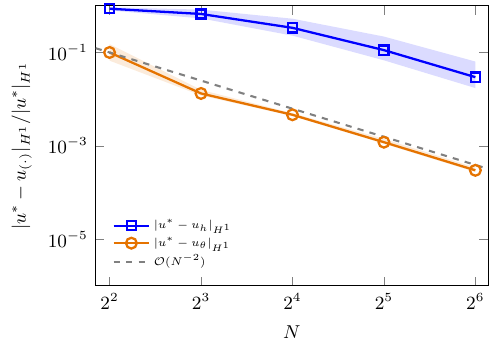}
        \caption{Quadratic B-splines (\(p=2\)).}
        \label{fig:conv-problem5-p2}
    \end{subfigure}
    \hfill
    \begin{subfigure}[b]{0.49\textwidth}
        \centering
        \includegraphics[width=\linewidth]{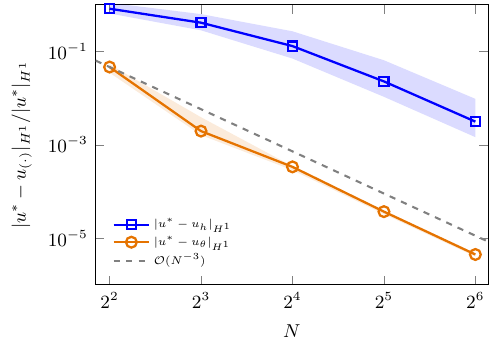}
        \caption{Cubic B-splines (\(p=3\)).}
        \label{fig:conv-problem5-p3}
    \end{subfigure}
    \caption{\textbf{Experiment~5.} Relative \(H^1\)-seminorm error versus \(N\) for
      the uniform and \(r\)-adaptive meshes (median over seeds and the held-out test
      parameters), trained with batch size~\(16\).}
    \label{fig:conv-problem5-p2p3}
\end{figure}

\begin{table}[t]
\centering
\small
\caption{\textbf{Experiment 5.} Median relative $H^1$ error $|u^\ast-u_{(\cdot)}|_{H^1}/|u^\ast|_{H^1}$ and median effectivity index $I_{\mathrm{eff}}=\eta/|u^\ast-u_{(\cdot)}|_{H^1}$ per level $N$, for the uniform ($u_h$) and adapted ($u_\theta$) meshes, at $p=2$ and $p=3$. Medians pool the four seeds and the held-out test parameters. The index divides the estimator by the unweighted \(H^1\) seminorm, while \(\eta\) controls the \(\varepsilon\)-weighted energy norm (see the text); \(\Ieff\) is reported only on \(\ell_\varepsilon\in[-1.75,-1.5]\).}
\label{tab:advdiff_summary}
\begin{adjustbox}{max width=\textwidth}
\begin{tabular}{c cccc @{\hspace{1.0em}} cccc}
\toprule
& \multicolumn{4}{c}{$p=2$} & \multicolumn{4}{c}{$p=3$} \\
\cmidrule(lr){2-5} \cmidrule(lr){6-9}
& \multicolumn{2}{c}{$H^1_{\mathrm{rel}}$} & \multicolumn{2}{c}{$I_{\mathrm{eff}}$} & \multicolumn{2}{c}{$H^1_{\mathrm{rel}}$} & \multicolumn{2}{c}{$I_{\mathrm{eff}}$} \\
\cmidrule(lr){2-3} \cmidrule(lr){4-5} \cmidrule(lr){6-7} \cmidrule(lr){8-9}
$N$ & $u_h$ & $u_\theta$ & $u_h$ & $u_\theta$ & $u_h$ & $u_\theta$ & $u_h$ & $u_\theta$ \\
\midrule
4 & $8.60{\times}10^{-1}$ & $9.68{\times}10^{-2}$ & $2.40$ & $7.66$ & $8.20{\times}10^{-1}$ & $6.82{\times}10^{-2}$ & $3.51$ & $6.85$ \\
8 & $6.64{\times}10^{-1}$ & $1.13{\times}10^{-2}$ & $1.80$ & $9.64$ & $4.12{\times}10^{-1}$ & $2.12{\times}10^{-3}$ & $2.30$ & $8.01$ \\
16 & $3.35{\times}10^{-1}$ & $4.29{\times}10^{-3}$ & $1.55$ & $6.18$ & $1.30{\times}10^{-1}$ & $3.15{\times}10^{-4}$ & $2.07$ & $4.14$ \\
32 & $1.12{\times}10^{-1}$ & $1.06{\times}10^{-3}$ & $1.55$ & $5.70$ & $2.28{\times}10^{-2}$ & $3.51{\times}10^{-5}$ & $2.07$ & $3.72$ \\
64 & $2.94{\times}10^{-2}$ & $2.63{\times}10^{-4}$ & $1.65$ & $5.73$ & $3.13{\times}10^{-3}$ & $4.35{\times}10^{-6}$ & $1.97$ & $3.70$ \\
\bottomrule
\end{tabular}
\end{adjustbox}
\end{table}

\begin{figure}[htbp]
    \centering
    \includegraphics[width=\linewidth]{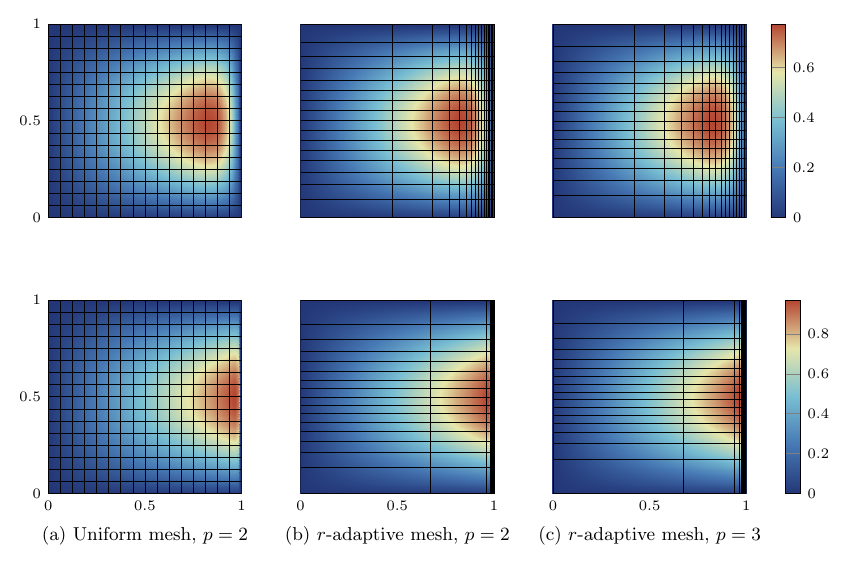}
    \caption{\textbf{Experiment~5.} Predicted meshes at \(N=16\) over the solution field,
      for two representative values of \(\bs\nu\) (one per row): upper panels~(a)--(c)
      correspond to \(\bs\nu=(-1.58,\,0.5)\) and lower panels~(a)--(c) to \(\bs\nu=(-1.95,\,2.0)\).}
    \label{fig:P5_solutions}
\end{figure}

\begin{figure}[htbp]
    \centering
    \includegraphics[width=\linewidth]{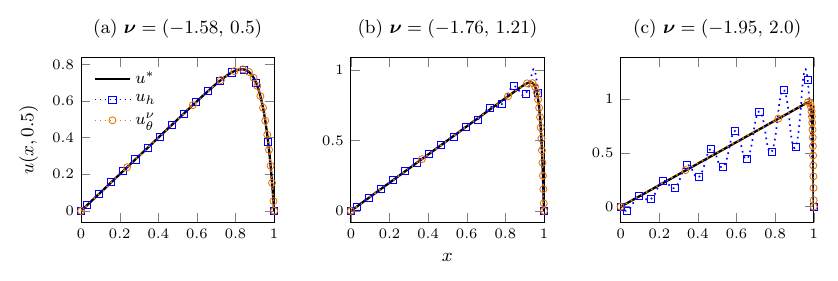}
    \caption{\textbf{Experiment~5.} Cross-section of the solution at \(y=0.5\)
      and \(N=16\) for \(p=2\), for three representative values of \(\bs\nu\),
      shown in panels~(a)--(c). Each panel compares the exact solution with
      its uniform and \(r\)-adaptive approximations; markers indicate the
      Greville abscissae of each mesh. The uniform mesh oscillates near
      \(x=1\), while the adapted mesh resolves the layer.}
    \label{fig:P5_cross_section}
\end{figure}

\begin{figure}[htbp]
    \centering
    \begin{subfigure}[t]{0.48\textwidth}
        \centering
        \includegraphics[scale=1]{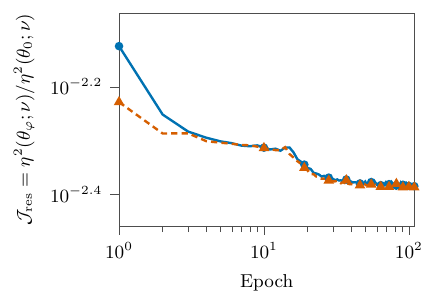}
        \caption{Quadratic B-splines (\(p=2\)).}
        \label{fig:loss_advdiff_p2}
    \end{subfigure}
    \hfill
    \begin{subfigure}[t]{0.48\textwidth}
        \centering
        \includegraphics[scale=1]{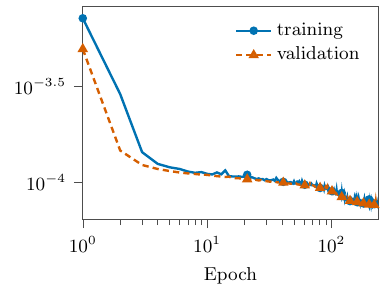}
        \caption{Cubic B-splines (\(p=3\)).}
        \label{fig:loss_advdiff_p3}
    \end{subfigure}
    \caption{\textbf{Experiment~5.} Training and validation losses \(\mathcal{J}_{\rm res}\) versus Adam epoch at the fixed refinement level \(N=4\), for (a)~\(p=2\) and (b)~\(p=3\).}
    \label{fig:loss_advdiff}
\end{figure}

\section{Conclusions}

\label{sec:conclusions}

We have introduced a neural \(r\)-adaptive IGA algorithm in which the physical solution is always computed by a standard Galerkin solve, while a neural network relocates the interior knots through a residual-based objective; the network does not replace the solver.

The central idea is the choice of the loss. A global strong-form residual, as minimized by PINNs, controls a norm stronger than the energy (\(H^1\)) error. Weighting the element residuals by the local mesh size and adding the interface flux jumps---the classical a posteriori construction---yields a computable estimator of the energy error, and this estimator is our training loss. In the coercive, shape-regular, conforming regime, the estimator is reliable and locally efficient (Proposition~\ref{prop:dual_reliability}, Theorem~\ref{thm:local_efficiency}); outside that regime it remains a well-defined mesh-quality functional. Since the loss requires no energy minimization principle, it extends differentiable \(r\)-adaptivity beyond Ritz formulations, which are restricted to symmetric coercive problems~\cite{aballay2025radaptivefiniteelementmethod,magueresse2025energy}, and covers the indefinite and advection-dominated problems tested here. Exact mesh gradients are obtained by the discrete adjoint, which reverse-mode AD applies to the linear solve at the cost of one extra solve.

In the parametric setting, the network maps each problem parameter to a knot-density function and predicts an adapted mesh in a single forward pass, with no per-instance optimization. Since the output is a continuous density rather than a fixed-size vector of knot positions, the network is independent of the element count and produces an admissible mesh at any refinement level, coarser or finer than those seen in training; this is what enables the coarse-to-fine continuation used throughout.

We observe in the numerical results that the adapted meshes concentrate degrees of freedom near singularities, interfaces, and boundary layers, improving accuracy at fixed cost. The main limitation is structural: each tensor-product knot line spans the whole domain, so localized refinement propagates along entire rows or columns, as seen in the re-entrant corner example (Section~\ref{subsec:exp4_lshape}). Two research directions follow naturally. First, coupling the present residual-driven \(r\)-adaptivity with hierarchical splines~\cite{giannelli2012thb,buffa2022mathematical} would remove the tensor-product restriction. Second, extending the framework to transient problems: within a time-stepping scheme, the network would predict a parameter- and time-dependent knot density, relocating the mesh as the solution features evolve---moving layers and traveling fronts being the natural targets---while the density representation keeps a single network across all time steps and refinement levels.

\section*{Acknowledgements}
Elias Car\'u has received funding from the European Union's Horizon Europe research and innovation programme under the Marie Sk\l{}odowska-Curie Action MSCA-DN-101119556 (IN-DEEP). David Pardo and Judit Mu\~noz-Matute have also received funding from the European Union's Horizon Europe research and innovation programme under the Marie Sk\l{}odowska-Curie Action MSCA-DN-101119556 (IN-DEEP), as well as from the Consolidated Research Group MATHMODE (IT1866-26) of the EHU given by the Department of Education of the Basque Government. David Pardo has also received funding from the following Research Projects/Grants: PID2023-146678OB-I00 funded by MICIU/AEI/10.13039/501100011033 and by FEDER, EU; BCAM Severo Ochoa accreditation of excellence CEX2021-001142-S funded by MICIU/AEI/10.13039/501100011033; Basque Government through the BERC 2022-2025 program; RUL-ET (KK-2024/00086), funded by the Basque Government through ELKARTEK; BCAM-IKUR-UPV/EHU, funded by the Basque Government IKUR Strategy and by the European Union NextGenerationEU/PRTR. Judit Mu\~noz-Matute has also received funding from the Research Project PID2023-146668OA-I00 and the grant RYC2023-045172-I funded by MICIU/AEI/10.13039/501100011033.
\printbibliography

\end{document}